\documentclass[a4paper,reqno]{amsart} 

\usepackage{amssymb}
\usepackage[mathcal]{euscript} 

\usepackage{tikz-cd} 

\usepackage{hyperref}

\newcommand{\bydef}{:=}

\newcommand{\id}{\mathrm{id}}

\newcommand{\trace}{\mathrm{tr}}

\newcommand{\op}{\mathrm{op}}

\newcommand{\cA}{\mathcal{A}}
\newcommand{\cB}{\mathcal{B}} 
\newcommand{\cC}{\mathcal{C}}
\newcommand{\cD}{\mathcal{D}} 

\newcommand{\cH}{\mathcal{H}} 

\newcommand{\cJ}{\mathcal{J}} 
\newcommand{\cK}{\mathcal{K}}
\newcommand{\cL}{\mathcal{L}} 
\newcommand{\cM}{\mathcal{M}}
\newcommand{\cN}{\mathcal{N}}

\newcommand{\cS}{\mathcal{S}} 
\newcommand{\cT}{\mathcal{T}}

\newcommand{\cW}{\mathcal{W}}

\newcommand{\frg}{{\mathfrak g}}

\newcommand{\frs}{{\mathfrak s}}

\newcommand{\ZZ}{\mathbb{Z}}

\newcommand{\FF}{\mathbb{F}} 

\DeclareMathOperator{\Hom}{\mathrm{Hom}}
\DeclareMathOperator{\End}{\mathrm{End}}

\DeclareMathOperator{\AAut}{\mathbf{Aut}}
\DeclareMathOperator{\Der}{\mathrm{Der}}

\DeclareMathOperator{\Cent}{\mathrm{Cent}}

\DeclareMathOperator{\Mat}{\mathrm{Mat}}

\newcommand{\ad}{\mathrm{ad}}

\newcommand{\frsl}{{\mathfrak{sl}}}

\newcommand{\frso}{{\mathfrak{so}}}
\newcommand{\frpsl}{{\mathfrak{psl}}}
\newcommand{\frgl}{{\mathfrak{gl}}}

\newcommand{\frosp}{{\mathfrak{osp}}}

\newcommand{\frstu}{{\mathfrak{stu}}}

\newcommand{\tri}{\mathfrak{tri}}

\newcommand{\frel}{\mathfrak{el}}
\newcommand{\frbr}{\mathfrak{br}}
\newcommand{\frinstrl}{\mathfrak{instrl}}

\newcommand{\SLs}{\mathbf{SL}}

\providecommand{\espan}[1]{\textup{span}\left\{ #1\right\}}

\providecommand{\ptriple}[1]{\boldsymbol{(}#1\boldsymbol{)}}

\newenvironment{romanenumerate} 
{\begin{enumerate}

}{\end{enumerate}}

\newcommand{\subo}{_{\bar 0}} 
\newcommand{\subuno}{_{\bar 1}}

\newcommand{\cE}{\mathcal{E}}

\newcommand{\Cl}{\mathfrak{Cl}} 

\newcommand{\bup}{\textup{b}} 
\newcommand{\boldbup}{\textup{\textbf{b}}}
\newcommand{\nup}{\textup{n}}

\newcommand{\hup}{\textup{h}}

\newcommand{\sVec}{\mathsf{sVec}}
\newcommand{\balpha}{\boldsymbol{\alpha}}

\newcommand{\Repap}{\mathsf{Rep}\,\balpha_p}
\newcommand{\Ver}{\mathsf{Ver}}
\newcommand{\Lss}{\mathcal{L}^{\textup{ss}}}

\newtheorem{theorem}{Theorem}[section]
\newtheorem{proposition}[theorem]{Proposition}
\newtheorem{lemma}[theorem]{Lemma}
\newtheorem{corollary}[theorem]{Corollary}

\theoremstyle{definition} 
\newtheorem{definition}[theorem]{Definition}
\newtheorem{example}[theorem]{Example}

\theoremstyle{remark} \newtheorem{remark}[theorem]{Remark}
\numberwithin{equation}{section}

\def\hregleta{\hrule height .5pt}
\def\hreglon{\hrule height 1pt}
\def\vreglon{\vrule height 12pt width1pt depth 4pt}

\def\hreglonfill{\leaders\hreglon\hfill}

\def\hregletafill{\leaders\hregleta\hfill}

\begin{document}

\title{$J$-ternary algebras, structurable algebras, and Lie superalgebras}

\author[I.~Cunha]{Isabel Cunha}
\address{Departamento de
Matem\'{a}tica e Centro de Matem\'atica e Aplica\c{c}\~oes da Universidade da Beira Interior,
Universidade da Beira Interior, 6201-001 Covilh\~{a}, Portugal}
\email{icunha@ubi.pt} 
\thanks{I.C. has been supported by FCT (Funda\c{c}\~ao para a Ci\^encia e
a Tecnologia, Portugal), research project UIDB/00212/2020 of CMA-UBI (Centro de 
Matem\'atica e Aplica\c{c}\~oes, Universidade da Beira Interior, Portugal).}

\author[A.~Elduque]{Alberto Elduque} 
\address{Departamento de
Matem\'{a}ticas e Instituto Universitario de Matem\'aticas y
Aplicaciones, Universidad de Zaragoza, 50009 Zaragoza, Spain}
\email{elduque@unizar.es} 
\thanks{A.E. has been supported by grant
PID2021-123461NB-C21, funded by 
MCIN/AEI/10.13039/ 501100011033 and by
 ``ERDF A way of making Europe''; and by grant 
E22\_20R (Gobierno de Arag\'on).}




\begin{abstract}
A Lie superalgebra is attached to any finite-dimensional $J$-ternary algebra over an algebraically closed
field of characteristic $3$, using a process of semisimplification via tensor categories. Some of the
exceptional simple Lie algebras, specific of this characteristic, are obtained in this way from $J$-ternary
algebras coming from structurable algebras and, in particular,
a new magic square of Lie superalgebras is constructed, with entries depending on a pair of composition
algebras.
\end{abstract}

\maketitle

\section{Introduction}\label{se:intro}

In a recent work \cite{Kannan}, Arun Kannan showed how to obtain a Lie superalgebra starting with a 
Lie algebra
over an algebraically closed field $\FF$ of characteristic $3$ endowed with a nilpotent derivation 
$d$ such that $d^3=0$, using a 
semisimplification process of symmetric tensor categories, taking advantage that the
semisimplification of the  category of
representations of $\FF[X]/(X^3)$ is the Verlinde category $\Ver_3$, which is equivalent to the category of
finite-dimensional vector superspaces over $\FF$. In characteristic $p>3$ still a suitable subcategory
of the Verlinde category is equivalent to the category of vector superspaces, and this allows Kannan  
to get interesting exceptional Lie superalgebras also in characteristic
$5$ from classical Lie algebras endowed with a derivation $d$ such that $d^5=0$.

A natural class of Lie algebras over a field of characteristic $3$ equipped with nilpotent derivations of
degree $3$ is given by the Lie algebras with a subalgebra isomorphic to $\frsl_2$ and such that, as 
a module for this subalgebra, they are a sum of copies of the adjoint, the natural and the trivial modules.
In this case, the adjoint action of the element 
$F=\left(\begin{smallmatrix}0&0\\ 1&0\end{smallmatrix}\right)\in\frsl_2$ is a nilpotent derivation and
$\ad_F^3=0$.
These are the Lie algebras with a \emph{short $\SLs_2$-structure} in the sense of Vinberg \cite{Vinberg}. 
These Lie algebras, over fields of characteristic $\neq 2,3$, are coordinatized by the so called 
$J$-ternary algebras, introduced by 
Hein \cite{Hein_Lie} and Allison \cite{Allison_JTernary}. However, only minor changes are needed in
characteristic $3$.

Therefore, giving a $J$-ternary algebra $\cT$ over a field of characteristic $3$, we can attach canonically
a Lie algebra $\cL(\cT)$ with a short $\SLs_2$-structure (see \eqref{eq:L(T)}), which can be viewed as
a Lie algebra in the category of representations of $\FF[X]/(X^3)$, and then we can apply the
semisimplification process to obtain a Lie superalgebra $\Lss(\cT)$ (Corollary
\ref{co:ternary_Lie_super}).

\smallskip

The goal of this paper is to study the main examples of simple finite-dimensional $J$-ternary
algebras over an algebraically closed field of characteristic $3$, apply the process above, and check which
Lie superalgebras appear. 

\smallskip

The outcome is that the `prototypical' $J$-ternary algebras give classical Lie superalgebras 
(\S\ref{ss:proto}), the $J$-ternary algebras obtained from simple structurable algebras of 
skew-dimension one give the new exceptional Lie superalgebras in 
\cite[Theorem 3.2]{Elduque_NewSimple3}, which are specific of characteristic $3$. For  the remaining
class of simple $J$-ternary algebras, i.e., those obtained from structurable algebras which
are tensor products of a Cayley algebra and another unital composition algebra, we get two
of the exceptional Lie superalgebras in the magic square of superalgebras in
 \cite{CunhaElduque_Extended}, as well as the exceptional Lie superalgebra $\frel(5;3)$ (notation
as in \cite{BouarroudjGrozmanLeites}).

\smallskip

The paper is organized as follows. Section \ref{se:JternarySL2} reviews the connection of $J$-ternary
algebras and Lie algebras with a short $\SLs_2$-structure, with the modifications needed to include
characteristic $3$. In particular, given any $J$-ternary algebra $\cT$ over a field of characteristic $\neq 2$,
a Lie algebra $\cL(\cT)$ with a short
$\SLs_2$-structure will be canonically defined in Theorem \ref{th:JTernarySL2}.

Section \ref{se:structurable} will show how to construct a $J$-ternary algebra out of a structurable
algebra containing a skew-symmetric element $s$ with invertible left multiplication $L_s$ (Theorem 
\ref{th:structurable_Jternary}). In the process, it will be proved that the Kantor Lie algebra 
$\cK(\cA,-)$ attached to a structurable algebra (\cite{Allison_Isotropic}) makes sense too 
in characteristic $3$ (Proposition \ref{pr:KantorLieAlgebra}), which is a result of independent interest.

Section \ref{se:tensor} is devoted to review the transition from a Lie algebra with a nilpotent
derivation of degree $3$ over a field of characteristic $3$ to a Lie superalgebra via tensor categories,
following the ideas in \cite{Kannan,DES,ElduqueEtingofKannan}. Then this will be applied to the Lie
algebra $\cL(\cT)$ attached in Section \ref{se:JternarySL2} to a $J$-ternary algebra and its derivation
$\ad_F$ above, thus obtaining a Lie superalgebra $\Lss(\cT)$.

In Section \ref{se:JternarySuper} it will be shown that if $\cT$ is a simple $J$-ternary algebra of
`prototypical type', then the Lie superalgebra $\Lss(\cT)$ is either a projective special linear Lie 
superalgebra or an orthosymplectic Lie superalgebra (Theorem \ref{th:proto_super}). The situation
for the simple $J$-ternary algebras obtained from simple structurable algebras other than the tensor 
products of a Cayley algebra and a unital composition algebra will be dealt with too.

Finally, Section \ref{se:magic} will treat the case missed in Section \ref{se:JternarySuper}. Here
three exceptional simple Lie superalgebras specific of characteristic $3$ appear: 
$\frg(3,3)$, $\frel(5;3)$ and $\frg(6,6)$. Moreover,
putting together the Lie superalgebras constructed from $J$-ternary algebras obtained from the
structurable algebras given by a tensor product of two unital composition algebras, we get a new
`magic square' of Lie superalgebras in \S\ref{ss:magic}.

\bigskip

Immediately after posting our results in arXiv, we learned that Michiel Smet had arrived to the same 
results independently. The main differences with his work \cite{Smet} are that our treatment of the
`prototypical type' in Section \ref{se:JternarySuper} is a bit more concrete, dealing directly with
classical Lie algebras instead of suitable algebras of matrices over a composition algebra; and that
in our treatment of the $J$-ternary algebras coming from a structurable algebra obtained as the tensor
product of two composition algebras in Section \ref{se:magic}, the invertible 
skew-symmetric element $s$ has no extra restrictions. We thank Michiel Smet for sharing his results
with us even before posting them to arXiv.

\bigskip

\emph{All the algebras considered will be finite-dimensional and defined over a ground field 
$\FF$ of characteristic not $2$, unless otherwise stated.}

\bigskip

\section{\texorpdfstring{$J$}{J}-ternary algebras and short  
\texorpdfstring{$\SLs_2$}{SL2}-structures on Lie algebras}\label{se:JternarySL2}

This section is devoted to review the two different related definitions of $J$-ternary algebras in 
the literature
and their role as `coordinate algebras' of certain Lie algebras.

\begin{definition}[{\cite{Hein_Lie,Hein_degree>2}}]\label{df:JTHein}
Let $\cT$ be a triple system  with ternary multiplication $xyz$. Then $\cT$ is said to be
a \emph{$J$-ternary algebra} if the following conditions hold:
\begin{subequations}
\begin{gather}
xy(uvz)=(xyu)vz+u(yxv)z+uv(xyz), \label{eq:JTHein1}\\
xyz-zyx=zxy-xzy, \label{eq:JTHein2}
\end{gather}
\end{subequations}
for any $x,y,z,u,v\in\cT$.
\end{definition}

\begin{definition}[{\cite{Allison_JTernary,AllisonBenkartGao}}]\label{df:JTAllison}
Let $\cJ$ be a unital Jordan algebra with multiplication $a\cdot b$, for $a,b\in\cJ$. Let $\cT$ be a 
unital special Jordan module for $\cJ$ with action $a\bullet x$ for $a\in \cJ$ and $x\in \cT$, so that
\[
(a\cdot b)\bullet x=\frac{1}{2}\Bigl(a\bullet(b\bullet x)+b\bullet(a\bullet x)\Bigr),
\] 
for any $a,b\in\cJ$ and $x\in \cT$.
Assume $\langle.\mid.\rangle:\cT\times \cT\rightarrow \cJ$ is a skew-symmetric bilinear map and 
$\ptriple{.,.,.}:\cT\times \cT\times \cT\rightarrow \cT$ is a trilinear product on $\cT$. Then $\cT$ is 
called a \emph{$\cJ$-ternary algebra} if the following axioms hold for any $a\in \cJ$ and $x,y,z,w,v\in \cT$:
\begin{subequations}
\begin{gather}
a\cdot\langle x\mid y\rangle =\dfrac{1}{2}\bigl(\langle a\bullet x\mid y\rangle 
+\langle x\mid a\bullet y\rangle\bigr),\label{eq:JTAllison1}
\\
a\bullet\ptriple{ x,y,z}=\ptriple{ a\bullet x,y,z}-\ptriple{ x,a\bullet y,z} +\ptriple{x,y,a\bullet z},
\label{eq:JTAllison2}
\\
\ptriple{x,y,z}=\ptriple{z,y,x}-\langle x\mid z\rangle\bullet y, \label{eq:JTAllison3}
\\
\ptriple{x,y,z}=\ptriple{y,x,z}+\langle x\mid y\rangle\bullet z, \label{eq:JTAllison4}
\\
\langle\ptriple{x,y,z}\mid w\rangle+\langle z\mid\ptriple{x,y,w}\rangle 
     =\langle x\mid\langle z\mid w\rangle\bullet y\rangle, \label{eq:JTAllison5}
\\
\ptriple{x,y,\ptriple{z,w,v}}=\ptriple{\ptriple{x,y,z},w,v}+
    \ptriple{z,\ptriple{y,x,w},v}+\ptriple{z,w,\ptriple{x,y, v}}. \label{eq:JTAllison6}
\end{gather}
\end{subequations}
\end{definition}

\smallskip

These two definitions are intimately related, as we will see in Theorem \ref{th:JTHein_Allison}. In the 
proof, a new type of triple systems will be used.

Actually, Yamaguti and Ono \cite{YamOno} defined a wide class
of triple systems: the $(\epsilon,\delta)$ Freudenthal-Kantor triple systems, which extend the 
classical Freudenthal triple systems and are useful tools in the construction of Lie algebras and
 superalgebras.

\begin{definition}[{\cite{YamOno}}]\label{df:edFKTS}
An $(\epsilon,\delta)$ \emph{Freudenthal-Kantor triple system} ($\epsilon,\delta$ are either $1$ or $-1$)  
is a triple system $U$, with multiplication $xyz$, such that, if $L(x,y),K(x,y)\in\End_{\FF}(U)$ are 
defined by
\begin{equation}\label{eq:LK}
\left\{\begin{aligned}
L(x,y)z&=xyz\\
K(x,y)z&=xzy-\delta yzx,
\end{aligned}\right.
\end{equation}
then
\begin{subequations}\label{eq:FK}
\begin{gather}
[L(u,v),L(x,y)]=L\bigl(L(u,v)x,y\bigr)+\epsilon L\bigl(x,L(v,u)y\bigr),\label{eq:FK1}\\
K\bigl(K(u,v)x,y\bigr)=L(y,x)K(u,v)-\epsilon K(u,v)L(x,y),\label{eq:FK2}
\end{gather}
hold for any $x,y,u,v\in U$.
\end{subequations}
\end{definition}

For $\epsilon=-1$ and $\delta=1$, these are the so called \emph{Kantor triple systems} (or 
generalized Jordan triple systems of second order \cite{Kan73}).

Let us show the close relationship between the $J$-ternary algebras and some particular 
$(1,1)$ Freudenthal-Kantor triple systems.

As in \cite{EldOku_dicyclic} we get the following properties.

\begin{lemma}\label{le:FK1}
Let $U$ be a triple system satisfying equation \eqref{eq:FK1}, with $\epsilon\in\{\pm1\}$, and define 
the endomorphisms $S(x,y)$ and $T(x,y)\in\End_\FF(U)$ by
\begin{equation}\label{eq:STs}
\begin{split}
S(x,y)&=L(x,y)+\epsilon L(y,x)\\
T(x,y)&=L(y,x)-\epsilon L(x,y).
\end{split}
\end{equation}
Then for any $u,v\in U$, $S(u,v)$ is a derivation of the triple system $(U,xyz)$, while $T(u,v)$ satisfies
\[
T(u,v)\bigl(xyz\bigr)=\bigl(T(u,v)x\bigr)yz-x\bigl(T(u,v)y\bigr)z+xy\bigl(T(u,v)z\bigr)
\]
for any $x,y,z\in U$. As a consequence, the following equations hold:
\begin{subequations}\label{subeq:STs}
\begin{align}
[S(u,v),L(x,y)]&=L\bigl(S(u,v)x,y\bigr)+L\bigl(x,S(u,v)y\bigr)\label{eq:[SL]}\\
[T(u,v),L(x,y)]&=L\bigl(T(u,v)x,y\bigr)-L\bigl(x,T(u,v)y\bigr)\label{eq:[TL]}\\
[S(u,v),T(x,y)]&=T\bigl(S(u,v)x,y\bigr)+T\bigl(x,S(u,v)y\bigr)\label{eq:[ST]}\\
[T(u,v),S(x,y)]&=-\epsilon T\bigl(T(u,v)x,y\bigr)+\epsilon T\bigl(x,T(u,v)y\bigr)\label{eq:[TS]}\\
[T(u,v),T(x,y)]&=-\epsilon S\bigl(T(u,v)x,y\bigr)+\epsilon S\bigl(x,T(u,v)y\bigr)\label{eq:[TT]}
\end{align}
\end{subequations}
\end{lemma}
\begin{proof}
The fact that $S(u,v)$ is a derivation and that $T(u,v)$ satisfies the equation above when applied to a
 product $xyz$ follow at once from \eqref{eq:FK1}. These conditions are equivalent to equations 
\eqref{eq:[SL]} and \eqref{eq:[TL]}. The other equations are obtained from these.
\end{proof}

\smallskip

\begin{remark} Let $U$ be a triple system satisfying equation \eqref{eq:FK1}, with 
$\epsilon\in\{\pm1\}$,  and denote by $L(U,U)$ (respectively $S(U,U)$, $T(U,U)$) the linear span of 
the operators $L(x,y)$ (respectively $S(x,y)$, $T(x,y)$) for $x,y\in U$. Lemma \ref{le:FK1} shows 
that $S(U,U)$ is a subalgebra of $L(U,U)$, and that the conditions $L(U,U)=S(U,U)+T(U,U)$, 
$[S(U,U),T(U,U)]\subseteq T(U,U)$ and $[T(U,U),T(U,U)]\subseteq S(U,U)$ hold. In particular, this 
shows that $S(U,U)\cap T(U,U)$ is an ideal of $L(U,U)$, and modulo this ideal we get a $\ZZ/2\ZZ$-graded 
Lie algebra.
\end{remark}

\smallskip

\begin{definition}\label{df:edFKTSspecial}
Let $U$ be an $(\epsilon,\delta)$ Freudenthal-Kantor triple system. Then $U$ is said to be 
\emph{special} in case
\begin{equation}\label{eq:special}
K(x,y)=\epsilon\delta L(y,x)-\epsilon L(x,y)
\end{equation}
holds for any $x,y\in U$.
\end{definition}

\smallskip

\begin{proposition}\label{pr:JT_special11}
The $J$-ternary algebras are exactly the special $(1,1)$ Freudenthal-Kantor triple systems.
\end{proposition}
\begin{proof}
First note that, for $\epsilon=1$, \eqref{eq:JTHein1} coincides with \eqref{eq:FK1}, while for 
$\epsilon=\delta=1$
\eqref{eq:JTHein2} can be rewritten as 
\[
K(x,z)=T(x,z)
\]
due to \eqref{eq:LK} and \eqref{eq:STs}. Also, for $\epsilon=\delta=1$, a Freudenthal-Kantor triple system
is special if and only if $K(x,y)=T(x,y)$ for any $x,y$, because of \eqref{eq:special}.

Therefore, if $(U,xyz)$ is a special $(1,1)$ Freudenthal-Kantor triple system, \eqref{eq:JTHein1} follows
from \eqref{eq:FK1}, while \eqref{eq:JTHein2} follows from \eqref{eq:special}.

Conversely, if $\cT$ is a $J$-ternary algebra then, with $\epsilon=\delta=1$, \eqref{eq:FK1} is equivalent
to \eqref{eq:JTHein1}, \eqref{eq:special} is equivalent to \eqref{eq:JTHein2}. Besides, \eqref{eq:JTHein1} 
implies
\[
T(u,v)(xyz)=\bigl(T(u,v)x\bigr)yz-x\bigl(T(u,v)y\bigr)z+xy\bigl(T(u,v)z\bigr).
\]
Interchanging $x$ and $z$ above and subtracting we get
\[
T(u,v)K(x,z)=K\bigl(T(u,v)x,z\bigr)-K(x,z)T(u,v)+K\bigl(x,T(u,v)z\bigr),
\]
which, because $K(x,z)=T(x,z)$ by \eqref{eq:JTHein2}, becomes
\begin{equation}\label{eq:KKs}
K(u,v)K(x,z)+K(x,z)K(u,v)=K\bigl(K(u,v)x,z\bigr)+K\bigl(x,K(u,v)z\bigr),
\end{equation}
for all $x,y,z,u,v\in\cT$. Using that 
\[
L(x,y)=\frac{1}{2}\bigl(S(x,y)-T(x,y)\bigr)=\frac{1}{2}\bigl(S(x,y)-K(x,y)\bigr),
\]
we compute, as in \cite[Proof of Proposition 3.5]{EldOku_dicyclic}:
\[
\begin{split}
L(y,x)&K(u,v)-K(u,v)L(x,y)\\
&=\frac{1}{2}\Bigl(S(y,x)-K(y,x)\Bigr)K(u,v)-\frac{1}{2}K(u,v)\Bigl(S(x,y)-K(x,y)\Bigr)\\
&=\frac{1}{2}[S(x,y),K(u,v)]+\frac{1}{2}\Bigl(K(x,y)K(u,v)+K(u,v)K(x,y)\Bigr)\\
&=-\frac{1}{2}[K(u,v),S(x,y)]+\frac{1}{2}\Bigl(K(x,y)K(u,v)+K(u,v)K(x,y)\Bigr)\\
&=K\bigl(K(u,v)x,y\bigr)\quad\text{because of \eqref{eq:[TS]} and \eqref{eq:KKs},}
\end{split}
\]
thus obtaining that $\cT$ also satisfies \eqref{eq:FK2}.
\end{proof}

\smallskip

Recall that given an associative algebra $\cA$, its associated Jordan algebra $\cA^{(+)}$ is the
algebra defined on the vector space $\cA$ with multiplication $a\cdot b=\frac{1}{2}(ab+ba)$ for 
$a,b\in\cA$.

\begin{theorem}\label{th:JTHein_Allison}
Let $\cJ$ be a unital Jordan algebra and $\cT$ a $\cJ$-ternary algebra (Definition \ref{df:JTAllison}). 
Then $\cT$ is a $J$-ternary algebra (Definition \ref{df:JTHein}).

Conversely, let $\cT$ be a $J$-ternary algebra, then the subspace $\cJ\bydef \FF\id+K(\cT,\cT)$ is
a Jordan subalgebra of $\End_\FF(\cT)^{(+)}$, where
$K(x,y)z=xzy-yzx$ for $x,y\in\cT$ and $K(\cT,\cT)=\espan{K(x,y)\mid x,y\in\cT}$, and $\cT$ becomes a
$\cJ$-ternary algebra with the natural action of $J$ on $\cT$: $a\bullet x=a(x)$, and the multilinear maps:
\[
\begin{split}
\ptriple{.,.,.}:\cT\times \cT\times \cT&\rightarrow \cT,\quad \ptriple{x,y,z}=xyz,\\
\langle .\mid .\rangle: \cT\times \cT&\rightarrow \cJ,\quad \langle x\mid y\rangle =-K(x,y).
\end{split}
\]
\end{theorem}
\begin{proof}
If $\cT$ is a $\cJ$-ternary algebra for a Jordan algebra $\cJ$, then with $xyz=\ptriple{x,y,z}$, equation
\eqref{eq:JTAllison6} becomes \eqref{eq:JTHein1}, while equations \eqref{eq:JTAllison3} and 
\eqref{eq:JTAllison4} give:
\[
xyz-zyx=-\langle x\mid z\rangle\bullet y=zxy-xzy,
\]
which gives \eqref{eq:JTHein2}.

Conversely, let $\cT$ be a $J$-ternary algebra. Then the subspace $\cJ=\FF \id+K(\cT,\cT)$ is a Jordan
subalgebra of $\End_\FF(\cT)^{(+)}$ because of \eqref{eq:KKs}, with multiplication 
$f\cdot g=\frac{1}{2}(fg+gf)$, and $\cT$ becomes a special
Jordan module with $a\bullet x=a(x)$ for $a\in\cJ\subseteq\End_\FF(\cT)$ and $x\in\cT$. 
Now, with $\ptriple{x,y,z}=xyz$, and
$\langle x\mid y\rangle =-K(x,y)\in\cJ$, \eqref{eq:JTAllison6} is just \eqref{eq:JTHein1},
 \eqref{eq:JTAllison3} is the definition of $\langle x\mid z\rangle$ and \eqref{eq:JTAllison4} follows
from \eqref{eq:JTHein2}. Also, \eqref{eq:KKs} gives \eqref{eq:JTAllison1}, and equation \eqref{eq:[TL]}
gives \eqref{eq:JTAllison2} because $K(x,y)=T(x,y)$ for any $x,y$, as $\cT$ is a special $(1,1)$ 
Freudenthal-Kantor triple system. 
Finally, \eqref{eq:JTHein1} gives, for any $x,y,z,w,u$:
\[
\begin{split}
L(x,y)K(z,w)u&=L(x,y)\bigl(zuw-wuz)\\
&=K\bigl(L(x,y)z,w\bigr)u+K\bigl(z,L(x,y)w\bigr)u+K(z,w)L(y,x)u,
\end{split}
\]
so that
\[
K\bigl(L(x,y)z,w\bigr)+K\bigl(z,L(x,y)w\bigr)=L(x,y)K(z,w)-K(z,w)L(y,x).
\]
Hence, \eqref{eq:FK2} (with $\epsilon=1$) can be written as follows:
\[
K\bigl(K(z,w)y,x\bigr)=K\bigl(L(x,y)z,w\bigr)+K\bigl(z,L(x,y)w\bigr),
\]
which gives \eqref{eq:JTAllison5}, because of the skew-symmetry of $K$.
\end{proof}

\smallskip

It is time to provide the most important class of examples of $J$-ternary algebras (see \cite[Proposition 4.1]{BCCE} and references there in).

\begin{example}[Prototypical example]\label{ex:prototypical}
Let $(\cA,*)$ be a unital associative algebra with involution and let  $\cT$ be a left $\cA$-module 
endowed with a skew-hermitian form 
$\hup:\cT\times\cT\rightarrow \cA$. Then $\cT$ is a $\cJ$-ternary algebra, where $\cJ=\cH(\cA,*)$ is 
the Jordan
algebra of symmetric elements of $\cA$ relative to the involution $*$,  with 
the following operations: 
\begin{itemize}
\item $a\cdot b=\frac{1}{2}(ab+ba)$ for any $a,b\in\cJ=\cH(\cA,*)$,
\item $a\bullet x=ax$ for any $a\in\cJ$ and $x\in\cT$,
\item $\langle x\mid y\rangle = \hup(x,y)-\hup(y,x)$ for $x,y\in\cT$, and
\item $\ptriple{x,y,z}=\hup(x,y)z+\hup(z,x)y+\hup(z,y)x$ for $x,y,z\in\cT$.
\end{itemize}
\end{example}

\smallskip

The $J$-ternary algebras are tightly related with the Lie algebras with short $\SLs_2$-structures. This connection
has been dealt with in \cite{BCCE}, and references there in, assuming the characteristic of $\FF$ is
not $2$ nor $3$. However, due to \cite[Remark 2.3]{BCCE}, most of the arguments are valid too
in characteristic $3$. 

First the necessary definitions, following \cite{BCCE}:

\begin{definition}[{\cite[Definition 0.1]{Vinberg}}]\label{df:Sstructure}
Let $\mathbf{S}$ be a reductive algebraic group and let $\cL$ be a Lie algebra. An 
\emph{$\mathbf{S}$-structure} on $\cL$ is a homomorphism $\Phi:\mathbf{S}\rightarrow\AAut(\cL)$ 
from $\mathbf{S}$ into the algebraic group of automorphisms of $\cL$.

The $\mathbf{S}$-structure $\Phi:\mathbf{S}\rightarrow \AAut(\cL)$ on the Lie algebra $\cL$ is said to 
be \emph{inner} if there is a one-to-one Lie algebra homomorphism 
$\iota:\frs\hookrightarrow \cL$, where $\frs$ is the Lie algebra of $\mathbf{S}$,
 such that the following diagram commutes, where $\textup{d}\Phi$ denotes
the differential of $\Phi$: 
\begin{equation}\label{eq:iota}
\begin{tikzcd}
\frs\arrow[rd,"\textup{d}\Phi"']\arrow[r, hook,"\iota"]&\cL\arrow[d,"\ad"]\\
&\Der(\cL)
\end{tikzcd}
\end{equation}
\end{definition}

Note that if $\Phi:\mathbf{S}\rightarrow \AAut(\cL)$ is a non-inner $\mathbf{S}$-structure on the 
Lie algebra $\cL$, then we can take the split extension $\tilde\cL=\cL\oplus\frs$ as in 
\cite[p.~18]{Jacobson}, where $\frs$ acts on $\cL$ through $\textup{d}\Phi$, and this is endowed with a natural $\mathbf{S}$-structure. Hence, it is not harmful to restrict to inner 
$\mathbf{S}$-structures.

\smallskip

\begin{definition}\label{df:shortA1}
An $\SLs_2$-structure $\Phi:\SLs_2\rightarrow \AAut(\cL)$ on a Lie algebra $\cL$ is said to be \emph{short} if $\cL$ decomposes, as a module for $\SLs_2$ via $\Phi$, into a direct sum of
copies of the adjoint, natural, and trivial modules.
\end{definition}

Therefore, the isotypic decomposition of $\cL$ allows us to describe $\cL$ as follows:
\begin{equation}\label{eq:shortA1_isotypic}
\cL=\bigl(\frsl(V)\otimes \cJ\bigr)\oplus \bigl(V\otimes\cT\bigr)\oplus\cD,
\end{equation}
for vector spaces $\cJ$, $\cT$, and $\cD$, where $V$ is the natural two-dimensional representation of 
$\SLs_2\simeq\SLs(V)$. The action of $\SLs_2$ is given by the adjoint action of $\SLs_2$ on 
$\frsl(V)$, its natural action on $V$, and the trivial action on $\cJ$, $\cT$ and $\cD$. 
The subspace $\cD$, being the subspace of fixed elements by $\SLs_2$, is a subalgebra of $\cL$.

\smallskip

\begin{lemma}\label{le:sl2}
Let $V$ be a two-dimensional vector space.
\begin{romanenumerate}
\item 
The space $\Hom_{\frsl(V)}\bigl(\frsl(V)\otimes\frsl(V),\frsl(V)\bigr)$ of $\frsl(V)$-invariant linear maps  
$\frsl(V)\otimes\frsl(V)\rightarrow\frsl(V)$ is spanned by the (skew-symmetric) Lie bracket:
\[
f\otimes g\mapsto [f,g]=fg-gf.
\]
\item 
The space $\Hom_{\frsl(V)}\bigl(\frsl(V)\otimes\frsl(V),\FF\bigr)$ is spanned by the trace map:
\[
f\otimes g\mapsto \trace(fg).
\]
\item 
The space $\Hom_{\frsl(V)}\bigl(\frsl(V)\otimes V,V\bigr)$ is spanned by the natural action:
\[
f\otimes v\mapsto f(v).
\]
\item 
The space $\Hom_{\frsl(V)}\bigl(V\otimes V,\FF\bigr)$ is one-dimensional. Its nonzero elements
are of the form
\[
u\otimes v\mapsto (u\mid v),
\]
for a nonzero skew-symmetric bilinear form $(.\mid .)$ on $V$. 
\item
The space $\Hom_{\frsl(V)}\bigl(V\otimes V,\frsl(V)\bigr)$ is one-dimensional. Once a nonzero
skew-symmetric bilinear form $(.\mid .)$ is fixed on $V$, this subspace is spanned by the following symmetric map:
\[
u\otimes v\mapsto \gamma_{u,v}\bigl(: w\mapsto (u\mid w)v+(v\mid w)u\bigr).
\]
\item
The spaces $\Hom_{\frsl(V)}\bigl(\frsl(V)\otimes\frsl(V),V\bigr)$, 
$\Hom_{\frsl(V)}\bigl(\frsl(V)\otimes V,\frsl(V)\bigr)$,  \\
$\Hom_{\frsl(V)}\bigl(\frsl(V)\otimes V,\FF\bigr)$,  and
$\Hom_{\frsl(V)}\bigl(V\otimes V,V\bigr)$ are all trivial.
\end{romanenumerate}

Moreover, $\Hom_{\frsl(V)}$ may be replaced by $\Hom_{\SLs(V)}$ all over.
\end{lemma}
\begin{proof}
This is well known if the characteristic of $\FF$ is $\neq 2,3$ (see, e.g., 
\cite[Lemma 2.1]{EldOku_dicyclic}), but it remains valid in characteristic $3$ by 
\cite[Remark 2.3]{BCCE}.
\end{proof}

Let $\cL$ be a Lie algebra with an inner $\SLs_2$-structure and isotypic decomposition as in 
\eqref{eq:shortA1_isotypic}. The $\SLs_2$-structure being inner forces $\cJ$ to contain a distinguished element $1$, such that $\frsl(V)\otimes 1$ is the image of $\iota$ in 
\eqref{eq:iota}.

The $\SLs_2$-invariance or, equivalently, the $\frsl(V)$-invariance, of the Lie bracket in our Lie algebra 
$\cL$ gives, for any $f,g\in\frsl(V)$, $u,v\in V$, $a,b\in\cJ$, $x,y\in\cT$, and $D\in \cD$, the 
following conditions:
\begin{equation}\label{eq:shortA1bracket}
\begin{split}
[f\otimes a,g\otimes b]&=[f,g]\otimes a\cdot b+2\trace(fg)D_{a,b},\\
[f\otimes a,u\otimes x]&=f(u)\otimes a\bullet x,\\
[u\otimes x,v\otimes y]&=\gamma_{u,v}\otimes \langle x\mid y\rangle 
                                              +\bigl(u\mid v\bigr)d_{x,y},\\
[D,f\otimes a]&=f\otimes D(a),\\
[D,u\otimes x]&=u\otimes D(x),
\end{split}
\end{equation}
for suitable $\cD$-invariant bilinear maps
\begin{equation}\label{eq:shortA1maps}
\begin{split}
\cJ\times \cJ\rightarrow \cJ:&\quad (a,b)\mapsto a\cdot b\quad\text{(symmetric),}\\
\cJ\times \cJ\rightarrow \cD:&\quad (a,b)\mapsto D_{a,b}\quad\text{(skew-symmetric),}\\
\cJ\times \cT\rightarrow \cT:&\quad (a,x)\mapsto a\bullet x,\\
\cT\times \cT\rightarrow \cJ:&\quad (x,y)\mapsto \langle x\mid y\rangle\quad
                     \text{(skew-symmetric),}\\
\cT\times \cT\rightarrow \cD:&\quad (x,y)\mapsto d_{x,y}\quad\text{(symmetric),}\\
\cD\times \cJ\rightarrow \cJ:&\quad (D,a)\mapsto D(a),\\
\cD\times \cT\rightarrow \cT:&\quad (D,x)\mapsto D(x),
\end{split}
\end{equation}
such that 
\begin{equation}\label{eq:D1a}
1\cdot a=a,\quad D_{1,a}=0,\quad\text{and}\quad 1\bullet x=x,
\end{equation}
for any $a\in \cJ$ and $x\in\cT$.

The Jacobi identity on $\cL$ also shows that all these maps are invariant under the action of the Lie subalgebra 
$\cD$. The next result summarizes the properties of these maps:

\begin{theorem}\label{th:ShortA1Properties}
A Lie algebra $\cL$ is endowed with an inner short $\SLs_2$-structure if and only if there is a 
two-dimensional vector space $V$ such that $\cL$ is, up to isomorphism, the Lie algebra in \eqref{eq:shortA1_isotypic}, with Lie bracket given in \eqref{eq:shortA1bracket}, for suitable 
bilinear maps given in \eqref{eq:shortA1maps}, satisfying the following conditions:
\begin{itemize}
\item
$\cJ$ is a unital commutative algebra with the multiplication $a\cdot b$.
\item
The following equations hold for $a,b\in\cJ$ and $x\in \cT$:
\begin{equation}\label{eq:bullet}
1\bullet x=x,\quad (a\cdot b)\bullet x=\frac{1}{2}\bigl(a\bullet(b\bullet x)+b\bullet(a\bullet x)\bigr).
\end{equation}
\item
For any $a,b,c\in\cJ$ and $x,y,z\in\cT$, the following identities hold:
\begin{subequations}
\begin{gather}
D_{a,b}(c)=a\cdot (b\cdot c)-b\cdot (a\cdot c),\label{eq:Dab}\\
D_{a\cdot b,c}+D_{b\cdot c,a}+D_{c\cdot a,b}=0,\label{eq:Dabc}\\
4D_{a,b}(x)=a\bullet(b\bullet x)-b\bullet(a\bullet x),\label{eq:Dabx}\\
4D_{a,\langle x\mid y\rangle}=-d_{a\bullet x,y}+d_{x,a\bullet y},\label{eq:Dd}\\
2a\cdot \langle x\mid y\rangle=\langle a\bullet x\mid y\rangle +\langle x \mid a\bullet y\rangle,
\label{eq:axy}\\
d_{x,y}(a)=\langle a\bullet x\mid y\rangle -\langle x\mid a\bullet y\rangle,\label{eq:dxya}\\
d_{x,y}(z)-d_{z,y}(x)=\langle x\mid y\rangle \bullet z-\langle z\mid y\rangle\bullet x
  + 2\langle x\mid z\rangle\bullet y.\label{eq:dxyz}
\end{gather}
\end{subequations}
\item
For any $D\in\cD$, the linear endomorphism of $\cJ\oplus \cT$, given by $a+x\mapsto D(a)+D(x)$, for
$a\in\cJ$ and $x\in\cT$, is an even derivation of the $\ZZ/2\ZZ$-graded algebra with even part $\cJ$, odd part $\cT$, and multiplication given by the formula:
\begin{equation}\label{eq:J+T}
(a+x)\diamond(b+y)=\bigl(a\cdot b+\langle x\mid y\rangle\bigr)+\bigl(a\bullet y+b\bullet x\bigr),
\end{equation}
for $a,b\in\cJ$ and $x,y\in\cT$.\qed
\end{itemize}
\end{theorem}

\begin{proof}
The proof in \cite[Theorem 2.2]{EldOku_dicyclic}, where the characteristic is assumed to be $\neq 2,3$,
works word for word in characteristic $3$ too, with the only difference that \eqref{eq:Dabc} 
does not imply the algebra $\cJ$ to be a Jordan algebra. 

Note that if the characteristic is  $\neq 2,3$, \eqref{eq:Dabc} is equivalent
to the condition $D_{a\cdot a,a}=0$ for any $a\in\cJ$, and hence \eqref{eq:Dab}, with 
$b=a^{\cdot 2}$, gives 
$a^{\cdot 2}\cdot(c\cdot a)=(a^{\cdot 2}\cdot c)\cdot a$,  due to the commutativity of $a\cdot b$. 
It follows that $\cJ$ is a Jordan algebra in this case.  In characteristic $3$, we get that $\cJ$ only
satisfies the complete linearization of the Jordan identity, but not necessarily the Jordan identity.
\end{proof}

\smallskip

\begin{theorem}\label{th:JTernarySL2}
Let $\cL$ be a Lie algebra endowed with an inner short $\SLs_2$-structure, with isotypic decomposition
as in \eqref{eq:shortA1_isotypic} and bracket in \eqref{eq:shortA1bracket}. Then $\cT$ is a $J$-ternary
algebra with the triple product
\begin{equation}\label{eq:xyzdxy}
\ptriple{x,y,z}\bydef \frac{1}{2}\bigl(d_{x,y}(z)-\langle x\mid y\rangle\bullet z\bigr)
\end{equation}
for $x,y,z\in\cT$.

Conversely, let $\cT$ be a $J$-ternary algebra (that is, a special (1,1) Freudenthal-Kantor triple system)
 and consider the Jordan algebra $\cJ=\FF\id +K(\cT,\cT)$
as in Theorem \ref{th:JTHein_Allison}, then the direct sum
\begin{equation}\label{eq:L(T)}
\cL(\cT)\bydef \bigl(\frsl(V)\otimes \cJ)\oplus\bigl( V\otimes \cT\bigr)\oplus S(\cT,\cT)
\end{equation}
is a Lie algebra, with an inner $\SLs_2$-structure, with the Lie bracket given by the bracket in 
\begin{equation}\label{eq:STT}
S(\cT,\cT)=\espan{S(x,y)=L(x,y)+L(y,x)\mid x,y\in\cT}\leq\frgl(\cT)
\end{equation}
 and the following equations:
\begin{equation}\label{eq:L(T)bracket}
\begin{split}
[f\otimes a,g\otimes b]&=[f,g]\otimes \frac{1}{2}(ab+ba)\, +\, \frac{1}{2}\trace(fg)[a,b],\\
[f\otimes a,u\otimes x]&=f(u)\otimes a(x),\\
[u\otimes x,v\otimes y]&=\gamma_{u,v}\otimes K(x,y)\,+\, (u\vert v)S(x,y)\\
[\varphi,f\otimes a]&=f\otimes [\varphi,a],\\
[\varphi,u\otimes x]&=u\otimes \varphi(x),
\end{split}
\end{equation}
for any $f,g\in\frsl(V)$, $u,v\in V$, $a,b\in \cJ$, $x,y\in \cT$ and $\varphi\in S(\cT,\cT)$.
\end{theorem}

\begin{proof}
If $\cL$ is a Lie algebra endowed with an inner short $\SLs_2$-structure and with isotypic
decomposition as in \eqref{eq:shortA1_isotypic}, and we consider the triple product on $\cT$ given
by \eqref{eq:xyzdxy}, then for any $x,y,z\in\cT$ we get
\[
\begin{split}
\ptriple{x,y,z}-\ptriple{z,y,x}
  &=\frac{1}{2}\bigl(d_{x,y}(z)-d_{z,y}(x)-
      \langle x\mid y\rangle\bullet z+\langle z\mid y\rangle\bullet x\bigr)\\
 &=\langle x\mid z\rangle \bullet y
\end{split}
\]
because of \eqref{eq:dxyz}, and also
\[
\begin{split}
\ptriple{z,x,y}-\ptriple{x,z,y}
 &=\frac{1}{2}\bigl(d_{z,x}(y)-d_{x,z}(y)-
      \langle z\mid x\rangle\bullet y+\langle x\mid z\rangle\bullet y\bigr)\\
 &=\langle x\mid z\rangle \bullet y
\end{split}
\]
by the symmetry of $d_{x,y}$ and the skew-symmetry of $\langle x\mid y\rangle$. Hence, we obtain
\[
\ptriple{x,y,z}-\ptriple{z,y,x}=\langle x\mid z\rangle\bullet y=\ptriple{z,x,y}-\ptriple{x,z,y},
\]
which implies \eqref{eq:JTHein2}. Now, for $a\in\cJ$ and $x,y,z\in \cT$ we get:
\[
\begin{split}
a\bullet&\ptriple{x,y,z} -\ptriple{x,y,a\bullet z}\\
&=\frac{1}{2}\Bigl(a\bullet d_{x,y}(z)-d_{x,y}(a\bullet z)
   -a\bullet\bigl(\langle x\mid y\rangle\bullet z\bigr)+\langle x\mid y\rangle\bullet (a\bullet z)\Bigr)\\
&=\frac{1}{2}\Bigl(-d_{x,y}(a)\bullet z-4D_{a,\langle x\mid y\rangle}(z)\Bigr)\\
&\null\qquad\text{(using \eqref{eq:Dabx} and the fact that $d_{x,y}$ is a derivation)}\\[4pt]
&=\frac{-1}{2}\Bigl(\langle a\bullet x\mid y\rangle\bullet z-\langle x\mid a\bullet y\rangle\bullet z-d_{a\bullet x,y}(z)+d_{x,a\bullet y}(z)\Bigr)\\
&\null\qquad\text{(because of \eqref{eq:Dd} and \eqref{eq:dxya})}\\[4pt]
&=\ptriple{a\bullet x,y,z}-\ptriple{x,a\bullet y,z},
\end{split}
\]
thus obtaining the validity of \eqref{eq:JTAllison2}. The validity of \eqref{eq:JTHein1} follows from this,
 together with the invariance of $\ptriple{x,y,z}$ under the action of $\cD$:
\[
\begin{split}
\ptriple{x,y,\ptriple{u,v,z}}&=
     \frac{1}{2}d_{x,y}\bigl(\ptriple{u,v,z}\bigr)-\frac{1}{2}\langle x\mid y\rangle\bullet \ptriple{u,v,z}\\
 &=\frac{1}{2}\Bigl(\ptriple{d_{x,y}(u),v,z}+\ptriple{u,d_{x,y}(v),z}+\ptriple{u,v,d_{x,y}(z)}\\
  &\qquad -\ptriple{\langle x\mid y\rangle\bullet u,v,z}+\ptriple{u,\langle x\mid y\rangle\bullet v,z}
            -\ptriple{u,v,\langle x\mid y\rangle\bullet z}\Bigr)\\
 &=\ptriple{\ptriple{x,y,u},v,z}+\ptriple{u,\ptriple{y,x,v},z}+\ptriple{u,v,\ptriple{x,y,z}}.
\end{split}
\]

\smallskip

Conversely, let $\cT$ be a $J$-ternary algebra and consider the Jordan algebra $\cJ=\FF \id+K(\cT,\cT)$.
Define the anticommutative algebra $\cL(T)$ as in \eqref{eq:L(T)bracket}. Note first that this is well defined
because for $a,b\in\cJ\leq\End_\FF(\cT)^{(+)}$, the bracket $[a,b]$ lies in $S(\cT,\cT)$ by 
\eqref{eq:[TT]}, as $T(x,y)=K(x,y)$ for any $x,y\in\cT$ (see the proof of 
Proposition \ref{pr:JT_special11}). The maps in \eqref{eq:shortA1bracket} are the following:
\begin{gather*}
a\cdot b=\frac{1}{2}(ab+ba),\quad a\bullet x=a(x),\quad \langle x\mid y\rangle =K(x,y),\\
D_{a,b}=\frac{1}{4}[a,b],\quad d_{x,y}=S(x,y),
\end{gather*}
for $a,b\in\cJ$ and $x,y\in \cT$. All these maps are $S(\cT,\cT)$-invariant by \eqref{eq:[SL]}.

Let us check that the conditions in Theorem \ref{th:ShortA1Properties}
hold. Clearly $\cJ$ is a unital commutative algebra, as it is a Jordan subalgebra of $\End_\FF(\cT)^{(+)}$,
and \eqref{eq:bullet} holds. Also $S(\cT,\cT)$ acts by derivations of the algebra in \eqref{eq:J+T}.

Equation \eqref{eq:Dab} is clear, because for any $a,b,c\in\End_\FF(\cT)$ we have
\[
\begin{split}
a\cdot(b\cdot c)&-b\cdot(a\cdot c)\\
&=\frac{1}{4}\bigl(a(bc+cb)+(bc+cb)a-b(ac+ca)-(ac+ca)b\bigr)\\
&=\frac{1}{4}\bigl(abc+cba-bac-cab\bigr)\\
&=\frac{1}{4}\bigl([a,b]c-c[a,b]\bigr)=\frac{1}{4}[[a,b],c].
\end{split}
\]
Equation \eqref{eq:Dabc} is a consequence then of the Jacobi identity on $\End_\FF(\cT)$. Also,
for $a,b\in\cJ$ and $x\in\cT$ we get
\[
a\bullet(b\bullet x)-b\bullet(a\bullet x)=a(b(x))-b(a(x))=[a,b](x),
\]
thus obtaining \eqref{eq:Dabx}. Equations \eqref{eq:Dd}, \eqref{eq:axy} and \eqref{eq:dxya} follow,
respectively from equations \eqref{eq:[TT]}, \eqref{eq:KKs} and \eqref{eq:[TS]}. Finally, for $x,y,z\in\cT$
we compute:
\[
\begin{split}
\langle x\mid y\rangle\bullet z&-\langle z\mid y\rangle\bullet x+2\langle x\mid z\rangle\bullet y\\
 &=K(x,y)z-K(z,y)x+2K(x,z)y\\
 &=xzy-yzx-zxy+yxz+2xyz-2zyx\\
 &=(xzy-zxy)+(xyz+yxz)-(zyx+yzx)+(xyz-zyx)\\
 &=(xyz+yxz)-(zyx+yzx)\quad\text{(because of \eqref{eq:JTHein2})}\\
 &=S(x,y)z-S(z,y)x=d_{x,y}(z)-d_{z,y}(x),
\end{split}
\]
thus getting \eqref{eq:dxyz}.
\end{proof}


\bigskip

\section{Structurable algebras and \texorpdfstring{$J$}{J}-ternary algebras}\label{se:structurable}

This section is devoted to study the connection of structurable algebras with some $J$-ternary algebras.

Structurable algebras were defined in \cite{Allison78} over fields of characteristic $\neq 2,3$. The
definition over arbitrary fields, or even commutative rings, requires an extra condition 
\cite[\S 5]{AllisonFaulkner_Steinberg}. 

\begin{definition}\label{df:structurable}
A unital algebra with involution $(\cA,-)$  is a \emph{structurable algebra} if the following
two conditions hold:
\begin{subequations}
\begin{gather}
[V_{a,b},V_{c,d}]=V_{V_{a,b}(c),d}-V_{c,V_{b,a}(d)}\label{eq:str1}\\
(a-\overline{a},b,c)=(b,\overline{a}-a,c)\label{eq:str2}
\end{gather}
\end{subequations}
for any $a,b,c,d\in\cA$, where $(a,b,c)\bydef (ab)c-a(bc)$ is the associator of $a,b,c$, and 
$V_{a,b}(c)=(a\overline{b})c+(c\overline{b})a-(c\overline{a})b$.
\end{definition}

Note that \eqref{eq:str1} shows that any structurable algebra satisfies \eqref{eq:FK1} with $\epsilon=-1$
for the triple product
\begin{equation}\label{eq:triple_str}
\{a,b,c\}\bydef V_{a,b}(c)=(a\overline{b})c+(c\overline{b})a-(c\overline{a})b.
\end{equation}
Also, \eqref{eq:str2} may be written as $(s,x,y)=-(x,s,y)$ for any 
$s\in \cS=\{x\in\cA\mid \overline{x}=-x\}$ and $x,y\in\cA$. Applying 
the involution we also have $(x,y,s)=-(x,s,y)$. These equations will be referred to as the 
skew-alternativity of $(\cA,-)$.

Our first aim in this section is to show that the construction of the Lie algebra $\cK(\cA,-)$ in 
\cite{Allison_Isotropic}, defined over fields of characteristic $\neq 2,3$, makes sense too in characteristic
$3$.

Recall that a Lie triple system is a vector space $\cT$ endowed with a trilinear product $[a,b,c]$ 
satisfying the following equations:
\begin{subequations}
\begin{gather}
[u,u,v]=0,\label{eq:LTS1}\\
[u,v,w]+[v,w,u]+[w,u,v]=0,\label{eq:LTS2}\\
[a,b,[u,v,w]]=[[a,b,u],v,w]+[u,[a,b,v],w]+[u,v,[a,b,w]],\label{eq:LTS3}
\end{gather}
\end{subequations}
for any $a,b,u,v,w\in\cT$. Also, for any Lie triple system $\cT$, its \emph{standard embedding} is the
($\ZZ/2\ZZ$-graded) Lie algebra $L(\cT,\cT)\oplus\cT$, where $L(x,y)(z)=[x,y,z]$, $L(\cT,\cT)$ is the span of the operators
$L(x,y)$ (a Lie subalgebra of $\End_\FF(\cT)$ because of \eqref{eq:LTS3}), with the Lie bracket given by
\[
[A+x,B+y]=\bigl([A,B]+L(x,y)\bigr)+\bigr(A(y)-B(x)\bigr)
\]
for $x,y\in\cT$ and $A,B\in L(\cT,\cT)$.

\smallskip

As in \cite{Faulkner_triples}, given a structurable algebra, consider the triple product $[.,.,.]$ defined on
the direct sum of two copies of $\cA$: $KT(\cA)\bydef \cA_+\oplus \cA_{-}$, as follows:
\begin{equation}\label{eq:KTA}
\begin{split}
[a_\delta,b_{-\delta},c_\delta]&=\{a,b,c\}_\delta,\\ 
[a_\delta,b_\delta,c_\delta]&=0,\\
[a_{-\delta},b_{\delta},c_\delta]&=-\{b,a,c\}_\delta,\\
[a_\delta,b_{\delta},c_{-\delta}]&=\{a,c,b\}_\delta-\{b,c,a\}_\delta,
\end{split}
\end{equation}
for $a,b,c\in\cA$, $\delta=\pm$.

\begin{proposition}\label{pr:KTA}
Let $(\cA,-)$ be a structurable algebra. Then, with the triple product in \eqref{eq:KTA}, $KT(\cA)$ is 
a Lie triple system.
\end{proposition}
\begin{proof}
This is a consequence of \cite[Lemma 1.8 and Corollary 1.4]{Faulkner_triples}.
\end{proof}

\smallskip

\begin{remark}\label{re:KAA}
Let $(\cA,-)$ be a structurable algebra. Consider the triple product $\{a,b,c\}=V_{a,b}(c)$ 
in \eqref{eq:triple_str} and, as in \eqref{eq:LK}, the operators $K_{a,b}:c\mapsto \{a,c,b\}-\{b,c,a\}$.
A simple computation gives:
\begin{equation}\label{eq:Kab_str}
\begin{split}
K_{a,b}(c)&=\{a,c,b\}-\{b,c,a\}=V_{a,c}(b)-V_{b,c}(a)\\
 &=(a\overline{c})b+(b\overline{c})a-(b\overline{a})c
                   -(b\overline{c})a-(a\overline{c})b+(a\overline{b})c\\
 &=(a\overline{b}-b\overline{a})c=L_{a\overline{b}-b\overline{a}}(c),
\end{split}
\end{equation}
for any $a,b,c\in \cA$, where $L_x$ denotes the left multiplication by $x$. Hence, $K(\cA,\cA)$ (the
linear span of the operators $K_{a,b}$) is just $L_\cS$, where $\cS=\{a\in\cA\mid \overline{a}=-a\}=
\espan{x-\overline{x}\mid x\in\cA}$ is the subspace of skew-symmetric elements for the involution.

\cite[Corollary 1.4]{Faulkner_triples} shows that the following equation, which is equation (STS2) in
\cite{Faulkner_triples}, holds:
\begin{equation}\label{eq:KV}
K_{a,b}V_{u,v}+V_{v,u}K_{a,b}=K_{K_{a,b}(u),v}
\end{equation}
for any $a,b,u,v\in\cA$. Equations \eqref{eq:str1} and \eqref{eq:KV} show that $\cA$ is a 
$(-1,1)$ Freudenthal-Kantor triple system, i.e., a Kantor triple system.
\end{remark}

\smallskip

Given a structurable algebra $(\cA,-)$, consider the Lie triple system $\cT=KT(\cA)=\cA_+\oplus\cA_-$ in
\eqref{eq:KTA}. This Lie triple system $\cT$ is graded by $\ZZ$ with $\cT_{1}=\cA_+$ and 
$\cT_{-1}=\cA_-$. As a consequence, its standard embedding $\cL=L(\cT,\cT)\oplus\cT$ is $\ZZ$-graded
with:
\[
\cL_{\pm 1}=\cT_{\pm 1},\quad \cL_2=L(\cT_1,\cT_1),\quad \cL_{-2}=L(\cT_{-1},\cT_{-1}),\quad
\cL_0=L(\cT_1,\cT_{-1}).
\]
In what follows we will identify $\End_\FF(\cT)$ with the matrix algebra 
$\Mat_2\bigl(\End_\FF(\cA)\bigr)$. With this identification, we compute elements in $L(\cT,\cT)$:

\begin{itemize}
\item 
For $a,b,c\in \cA$ we have $[a_+,b_-,c_+]=\{a,b,c\}_+=V_{a,b}(c)_+$ and
$[a_+,b_-,c_-]=-[b_-,a_+,c_-]=-V_{b,a}(c)_-$. Hence, the operator $L(a_+,b_-)$, when considered
as an element in $\Mat_2\bigl(\End_\FF(\cA)\bigr)$, is
\[
L(a_+,b_-)=\begin{pmatrix} V_{a,b}&0\\ 0&-V_{b,a}\end{pmatrix}.
\]
We may consider, as in \cite{Allison78}, the linear endomorphism $\varepsilon$ on $\End_\FF(\cA)$ given
by \[
T^\varepsilon=T-L_{T(1)+\overline{T(1)}}.
\] 
For $x,y,z\in\cA$ we have
\[
\begin{split}
V_{x,y}(z)+V_{y,x}(z)&=(x\overline{y})z+(z\overline{y})x-(z\overline{x})y +
                                      (y\overline{x})z+(z\overline{x})y-(z\overline{y})x\\
      &=(x\overline{y}+y\overline{x})z=L_{x\overline{y}+y\overline{x}}(z).
\end{split}
\]
but $V_{x,y}(1)=x\overline{y}+\overline{y}x-\overline{x}y$, $\overline{V_{x,y}(1)}=
y\overline{x}+\overline{x}y-\overline{y}x$, so 
$x\overline{y}+y\overline{x}=V_{x,y}(1)+\overline{V_{x,y}(1)}$, and hence we get
$V_{x,y}+V_{y,x}=L_{V_{x,y}(1)+\overline{V_{x,y}(1)}}$, or 
\[
V_{x,y}^\varepsilon=-V_{y,x}.
\]
Therefore, $\varepsilon$ restricts to a linear automorphism of order $2$ of the
Lie algebra (\emph{inner structure Lie algebra})
$\mathfrak{instrl}(\cA,-)=\espan{V_{a,b}\mid a,b\in\cA}$. Equation \eqref{eq:str1} shows that
$\varepsilon$ becomes an order two automorphism of the Lie algebra $\mathfrak{instrl}(\cA,-)$.

With this notation, we get
\begin{equation}\label{eq:La+b-}
L(a_+,b_-)= \begin{pmatrix} V_{a,b}&0\\ 0&V_{a,b}^\varepsilon\end{pmatrix},
\end{equation}
and hence $\cL_0=L(\cT_1,\cT_{-1})$ is isomorphic to $\mathfrak{instrl}(\cA,-)$, by means of the
map $T\mapsto \begin{pmatrix} T&0\\ 0&T^\varepsilon\end{pmatrix}$.

\smallskip

\item
Again, for $a,b,c\in\cA$, we have 
$[a_+,b_+,c_-]=\{a,c,b\}_+-\{b,c,a\}_+=\bigl(L_{a\overline{b}-b\overline{a}}\bigr)(c)$
because of \eqref{eq:Kab_str}, which shows
\begin{equation}\label{eq:La+b+}
L(a_+,b_+)=\begin{pmatrix} 0&L_{a\overline{b}-b\overline{a}}\\ 0&0\end{pmatrix},
\end{equation}
so that $\cL_2$ is just $\begin{pmatrix} 0&L_{\cS}\\ 0&0\end{pmatrix}$, which can be identified with
$\cS$ by means of the map $s\mapsto \begin{pmatrix} 0&L_{s}\\ 0&0\end{pmatrix}$.

\smallskip

\item
In the same vein, 
\begin{equation}\label{eq:La-b-}
L(a_-,b_-)=\begin{pmatrix} 0& 0\\ L_{a\overline{b}-b\overline{a}}&0\end{pmatrix},
\end{equation}
and $\cL_{-2}$ can be identified too with $\cS$: 
$s\mapsto \begin{pmatrix} 0&0\\L_{s}&0\end{pmatrix}$.
\end{itemize}

\smallskip

Our next result shows that the Lie algebra $\cK(\cA,-)$ constructed in \cite{Allison_Isotropic} makes 
sense too assuming only that the characteristic is not $2$. It must be remarked that in characteristic 
$3$ it is
no longer true that $\mathfrak{instrl}(\cA,-)$ is the direct sum of $V_{\cA,1}$ and $\Der(\cA,-)$, a result
which is crucial in \cite{Allison_Isotropic} and that we cannot use here.

\begin{proposition}\label{pr:KantorLieAlgebra}
Let $(\cA,-)$ be a structurable algebra. Consider the Lie triple system $\cT=KT(\cA)$ in 
\eqref{eq:KTA} and let $\cL$ be its standard embedding. Then $\cL$ is isomorphic, as a $\ZZ$-graded
Lie algebra, to the algebra defined on the vector space
\[
\cK(\cA,-)=\cN^\sim\oplus\mathfrak{instrl}(\cA,-)\oplus \cN,
\]
where $\cN=\cA\times\cS=\{(x,s)\mid x\in\cA,\, s\in\cS\}$, $\cN^\sim$ is a copy of $\cN$, and 
the bracket is given by imposing that $\mathfrak{instrl}(\cA,-)$ is a subalgebra, together with
 the following formulas:
\begin{equation}\label{eq:KA-}
\begin{aligned}
\relax [T,(x,s)]&=\bigl(T(x),T^\delta(s)\bigr),\\ 
       [T,(x,s)^\sim]&=\bigl(T^\varepsilon(x),T^{\varepsilon\delta}(s)\bigr)^\sim,\\
[(x,s),(y,t)]&=(0,x\overline{y}-y\overline{x}), \\
     [(x,s)^\sim,(y,t)^\sim]&= (0,x\overline{y}-y\overline{x})^\sim,\\
[(x,s),(y,t)^\sim]&=-(tx,0)^\sim+\bigl(V_{x,y}+L_sL_t\bigr)+(sy,0),
\end{aligned}
\end{equation}
for $x,y\in\cA$, $s,t\in\cS$, $T\in\mathfrak{instrl}(\cA,-)$, where $T^\delta=T+R_{\overline{T(1)}}$, 
that is, $T^\delta(x)=T(x)+x\overline{T(1)}$ for any $T\in\mathfrak{instrl}(\cA,-)$ and $x\in\cA$.
\end{proposition}

Note that for $s,t\in\cS$ and $x\in\cA$, an easy computation gives
\[
\begin{split}
V_{st,1}(x)&-V_{s,t}(x)=(st)x+x(st)-x(\overline{st})-(s\overline{t})x-(x\overline{t})s+(x\overline{s})t\\
&=(st)x+x(st)-x(ts)+(st)x+(xt)s-(xs)t\\
&=2(st)x-(x,s,t)+(x,t,s)\\
&=2\bigl((st)x-(s,t,x)\bigr)\quad\text{by skew-alternativity}\\
&=2s(tx)=2L_sL_t(x),
\end{split}
\]
so that $L_sL_t=\frac{1}{2}(V_{st,1}-V_{s,t})$ is indeed in 
$\mathfrak{instrl}(\cA,-)$, and the bracket above is well defined. The $\ZZ$-grading on $\cK=\cK(\cA,-)$
is given by $\cK_{2}=\{(0,s)\mid s\in \cS\}$, $\cK_1=\{(x,0)\mid x\in\cA\}$, 
$\cK_0=\mathfrak{instrl}(\cA,-)$, $\cK_{-1}=\{(x,0)^\sim\mid x\in\cA\}$, and 
$\cK_{-2}=\{(0,s)^\sim\mid s\in\cS\}$.

\begin{proof}
Define the degree-preserving linear isomorphism $\varphi:\cL\rightarrow \cK$ as follows:
\begin{itemize}
\item 
$\varphi(x_+)=(x,0)$ and $\varphi(x_-)=(x,0)^\sim$, for $x\in\cA$,

\item
$\varphi\left(\begin{pmatrix} T&0\\ 0&T^\varepsilon\end{pmatrix}\right)=T$, for 
$T\in\mathfrak{instrl}(\cA,-)$,

\item 
$\varphi\left(\begin{pmatrix} 0&L_s\\ 0&0\end{pmatrix}\right)=(0,s)$ and
$\varphi\left(\begin{pmatrix} 0&0\\ L_s &0\end{pmatrix}\right)=(0,s)^\sim$, for $s\in\cS$.
\end{itemize}

Equations \eqref{eq:La+b-}, \eqref{eq:La+b+}, and \eqref{eq:La-b-} prove that 
$\varphi([X,Y])=[\varphi(X),\varphi(Y)]$ for $X,Y\in\cL_1\cup\cL_{-1}$. This is also trivially true for
$X\in\cL_1$ and $Y\in\cL_{-2}$, or $X\in\cL_{-1}$ and $Y\in\cL_2$.

For $s\in\cS$, \eqref{eq:Kab_str} gives $K_{s,1}(x)=(s\overline{1}-1\overline{s})x=2sx$, so we get
$K_{s,1}=2L_s$. Equations
\eqref{eq:KV} and \eqref{eq:Kab_str} give
\[
L_sV_{x,y}+V_{y,x}L_s=K_{sx,y}=L_{(sx)\overline{y}-y(\overline{sx})}.
\]
Hence, for $x,y\in\cA$ and $s\in \cS$ we get the following bracket in $\cL$:
\[
\left[\begin{pmatrix} V_{x,y}&0\\ 0&-V_{y,x}\end{pmatrix}, 
   \begin{pmatrix} 0&L_s\\ 0&0\end{pmatrix}\right]=
 \begin{pmatrix} 0& V_{x,y}L_s+L_sV_{y,x}\\ 0&0\end{pmatrix}=
 \begin{pmatrix} 0& L_{(sy)\overline{x}-x(\overline{sy})}\\ 0&0\end{pmatrix},
\]
but, since $\overline{s}=-s$, we compute
\[
\begin{split}
(sy)\overline{x}-x(\overline{sy})&= (sy)\overline{x}+x(\overline{y}s)\\
  &=(s,y,\overline{x})+s(y\overline{x})-(x,\overline{y},s)+(x\overline{y})s\\
 &=V_{x,y}(s)-(s,\overline{y},x)-s(\overline{y}x)+(s,\overline{x},y)+s(\overline{x}y)\\
 &\qquad +(s,y,\overline{x})+s(y\overline{x})-(s,x,\overline{y})\\
 &=V_{x,y}(s)+s(y\overline{x}+\overline{x}y-\overline{y}x)\\
 &\qquad +(s,\overline{x},y)+(s,y,\overline{x})-(s,\overline{y},x)-(s,x,\overline{y})\\
&=V_{x,y}(s)+s(\overline{x\overline{y}+\overline{y}x-\overline{x}y})\\
 &=V_{x,y}(s)+s\overline{V_{x,y}(1)}=V_{x,y}^\delta(s),
\end{split}
\]
because
\begin{multline*}
(s,\overline{x},y)+(s,y,\overline{x})-(s,\overline{y},x)-(s,x,\overline{y})\\
 = (s,\overline{x}-x,y)+(s,y,\overline{x}-x)-(s,\overline{y}-y,x)-(s,x,\overline{y}-y)=0,
\end{multline*}
by skew-alternativity.
Thus, for $x,y\in\cA$ and $s\in\cS$, we obtain the following bracket in $\cL$:
\[
\left[\begin{pmatrix} V_{x,y}&0\\ 0&-V_{y,x}\end{pmatrix}, 
   \begin{pmatrix} 0&L_s\\ 0&0\end{pmatrix}\right]=
 \begin{pmatrix} 0& V_{x,y}^\delta(s)\\ 0&0\end{pmatrix},
\]
which shows that $\varphi([X,Y])=[\varphi(X),\varphi(Y)]$ also for $X\in\cL_0$ and $Y\in \cL_2$.

In the same vein we compute, for $x,y\in\cA$ and $s\in\cS$,
\begin{multline*}
\left[\begin{pmatrix} V_{x,y}&0\\ 0&-V_{y,x}\end{pmatrix}, 
   \begin{pmatrix} 0&0\\ L_s&0\end{pmatrix}\right]
  =\begin{pmatrix} 0&0\\ -V_{y,x}L_s-L_sV_{x,y}&0\end{pmatrix}\\
   =\begin{pmatrix} 0&0\\ -L_{V_{y,x}(s)+s\overline{V_{y,x}(1)}}&0\end{pmatrix}
=\begin{pmatrix} 0&0\\ L_{V_{x,y}^\varepsilon(s)+s\overline{V_{x,y}^\varepsilon(1)}} & 0\end{pmatrix}
=\begin{pmatrix} 0&0\\ L_{V_{x,y}^{\varepsilon\delta}(s)}&0\end{pmatrix}
\end{multline*}

and this shows that $\varphi([X,Y])=[\varphi(X),\varphi(Y)]$ also for $X\in\cL_0$ and $Y\in \cL_{-2}$.

Finally, for $s,t\in\cS$ we get the bracket
\[
\left[\begin{pmatrix} 0&L_s\\ 0&0\end{pmatrix},\begin{pmatrix} 0&0\\ L_t&0\end{pmatrix}\right]
  =\begin{pmatrix} L_sL_t&0\\ 0&-L_tL_s\end{pmatrix}.
\]
Also, for any $x,y\in\cA$, 
$V_{x,1}(y)=xy+yx-y\overline{x}=V_{1,\overline{x}}(y)$, and hence
\[
(L_sL_t)^\varepsilon=\frac{1}{2}(V_{st,1}-V_{s,t})^\varepsilon
=\frac{1}{2}(-V_{1,st}+V_{t,s})=\frac{1}{2}(-V_{ts,1}+V_{t,s})=-L_tL_s.
\]
We conclude that $\varphi([X,Y])=[\varphi(X),\varphi(Y)]$ also for $X\in\cL_2$ and $Y\in \cL_{-2}$.
\end{proof}

\smallskip

\begin{remark}\label{re:KA-2}
A slightly different version of $\cK(\cA,-)$ is given in \cite[\S 6.4]{AllisonBenkartGao}, the bracket in
\eqref{eq:KA-} is substituted by
\begin{equation}\label{eq:KA-2}
\begin{aligned}
\relax [T,(x,s)]&=\bigl(T(x),T^\delta(s)\bigr),\\ 
       [T,(x,s)^\sim]&=\bigl(T^\varepsilon(x),T^{\varepsilon\delta}(s)\bigr)^\sim,\\
[(x,s),(y,t)]&=\bigl(0,2(x\overline{y}-y\overline{x})\bigr), \\
     [(x,s)^\sim,(y,t)^\sim]&=\bigl(0,2(x\overline{y}-y\overline{x})\bigr)^\sim,\\
[(x,s),(y,t)^\sim]&=-(tx,0)^\sim+\bigl(2V_{x,y}+L_sL_t\bigr)+(sy,0).
\end{aligned}
\end{equation}
As remarked in \cite[\S 6.4]{AllisonBenkartGao}, the map
\[
(y,t)^\sim + T + (x,s)\mapsto (y,\frac{1}{2}s)^\sim + T + (2x,2s)
\]
gives an isomorphism from the Lie algebra in \eqref{eq:KA-2} to the Lie algebra in \eqref{eq:KA-}.
\end{remark}

\bigskip

The next results are taken from \cite[\S 4.2]{BCCE}, but one has to be a bit careful as the characteristic
of the ground field here is just $\neq 2$. First, an easy lemma.

\begin{lemma}\label{le:5graded}
Let $\cL$ be a $5$-graded Lie algebra: $\cL=\cL_{-2}\oplus\cL_{-1}\oplus\cL_0\oplus\cL_1\oplus\cL_2$,
 and assume that there are elements $E\in\cL_2$ and 
$F\in\cL_{-2}$, such that the linear span of
$E$, $F$, and $H=[E,F]$ is a subalgebra isomorphic to $\frsl_2$, with 
$\cL_i=\{X\in\cL\mid [H,X]=iX\}$, for $i=-2,-1,0,1,2$.  (In particular $[H,E]=2E$ and 
$[H,F]=-2F$.) Then the following properties hold:
\begin{romanenumerate}
\item
The linear map $\ad_F\vert_{\cL_i}$ is one-to-one for $i=1,2$. Actually, the map 
$\cL_1\rightarrow \cL_{-1}$, $X\mapsto [F,X]$, is bijective with inverse $\cL_{-1}\rightarrow \cL_1$,
$Y\mapsto [E,Y]$; and the map $\cL_2\rightarrow \cL_{-2}$, $X\mapsto [F,[F,X]]$, is bijective
with inverse $\cL_{-2}\rightarrow\cL_2$, $Y\mapsto \frac{1}{4}[E,[E,Y]]$.

\item
Denote by $\frsl_2$ the subalgebra spanned by $H$, $E$, and $F$. Then we have
\[
\Cent_{\cL}(\frsl_2)=\{X\in\cL_0\mid [F,X]=0\}=\{X\in \cL_0\mid [E,X]=0\},
\]
where $\Cent_{\cL}(\frsl_2)$
denotes the centralizer in $\cL$ of the subalgebra $\frsl_2$. 

\item 
 The subalgebra $\cL_0$ splits as 
\[
\cL_0=\Cent_{\cL}(\frsl_2)\oplus [F,\cL_2].
\]

\item
The Lie algebra $\cL$ splits as
\[
\cL=\bigl(\cL_2\oplus [F,\cL_2]\oplus \cL_{-2}\bigr)\oplus \bigl(\cL_1\oplus\cL_{-1}\bigr)
\oplus \Cent_\cL(\frsl_2),
\]
and the subspace $\cL_2\oplus [F,\cL_2]\oplus \cL_{-2}=\sum_{X\in\cL_2}\espan{X,[F,X],[F,[F,X]]}$
is a sum of copies of the adjoint module for $\frsl_2$, the subspace 
$\cL_1\oplus\cL_{-1}=\sum_{X\in\cL_1}\espan{X,[F,X]}$ is a sum of copies of its natural two-dimensional 
module, and $\Cent_{\cL}(\frsl_2)$ is a sum of copies of the one-dimensional trivial module.

In particular the dimension of $\cL$ coincides with $3\dim\cL_2+2\dim\cL_1+\dim\Cent_{\cL}(\frsl_2)$.

\item
If the conditions $\cL_2=[\cL_1,\cL_1]$ and $\cL_0=[\cL_1,\cL_{-1}]$ hold, then we get
\begin{gather*}
\Cent_{\cL}(\frsl_2)=\espan{[X,[F,Y]]+[Y,[F,X]]\mid X,Y\in\cL_1},\\
[F,\cL_2]=\espan{[X,[F,Y]]-[Y,[F,X]]\mid X,Y\in\cL_1}.
\end{gather*}
\end{romanenumerate}
\end{lemma}

\begin{proof}
If $X\in\cL_i$, $i=1,2$, and $[F,X]=0$, then we get $0=[E,[F,X]]=[[E,F],X]=[H,X]=iX$, because $[E,X]$ 
belongs to $\cL_{i+2}=0$. For $X\in\cL_1$, this argument, together with the analogous argument
with $F$ changed by $E$ and $i$ by $-i$, shows that the linear maps 
$\ad_F\colon \cL_1\rightarrow \cL_{-1}$ and $\ad_E:\cL_{-1}\rightarrow\cL_1$ are bijective and one is
the inverse of the other. Also, for $X\in\cL_2$, we get
\begin{multline*}
[E,[E,[F,[F,X]]]]=[E,[F,[E,[F,X]]]]\\
=[E,[F,[H,X]]]=2[E,[F,X]]=2[H,X]=4X,
\end{multline*}
where we have used that $[F,X]$ lies in $\cL_0=\{Z\in\cL\mid [H,Z]=0\}$ and that $[E,X]=0$. 
Assertion (i) follows.

It is clear that $\Cent_{\cL}(\frsl_2)$ lies in the centralizer of $F$, or of $E$. On the other hand,
$\Cent_{\cL}(\frsl_2)$ lies in the centralizer of $H$, which is $\cL_0$. Now, for $X\in\cL_0$ with $[F,X]=0$
we get $0=[E,[F,X]]=[F,[E,X]]$, but $\ad_F\vert_{\cL_2}$ is one-to-one, so $[E,X]=0$ follows. This proves
(ii).

For (iii), if $X$ lies in the intersection $\Cent_{\cL}(\frsl_2)\cap [F,\cL_2]$, then there is an element
$Y\in\cL_2$ such that $X=[F,Y]$ and $[E,X]=0$, so that $0=[E,X]=[E,[F,Y]]=[H,Y]=2Y$, and hence $Y=0$,
and also $X=0$. On the other hand, for $X\in\cL_0$, an easy computation gives:
\[
\begin{split}
[F,[F,[E,X]]]&=-[F,[H,X]]+[F,[E,[F,X]]]\\
   &=[F,[E,[F,X]]]\quad\text{because $[H,X]=0$,}\\
 &=-[H,[F,X]]+[E,[F,[F,X]]]\\
 &=2[F,X]\quad\text{because $[F,X]\in\cL_{-2}$ and $[F,[F,X]]\in\cL_{-4}=0$.}
\end{split}
\]
Hence, we have
\[
X=\left(X-\frac{1}{2}[F,[E,X]]\right)+ \frac{1}{2}[F,[E,X]]\in\Cent_{\cL}(\frsl_2)+[F,\cL_2],
\]
and (iii) follows. 

The assertion in (iv) follows at once from the previous ones.

Finally, in  case $\cL_2=[\cL_1,\cL_1]$ and $\cL_0=[\cL_1,\cL_{-1}]$, we get
\[
\cL_0=[\cL_1,\cL_{-1}]=[\cL_1,[F,\cL_1]]=\espan{[X,[F,Y]]\mid X,Y\in\cL_1},
\]
but for $X,Y\in\cL_1$ we have
\[
[X,[F,Y]]-[Y,[F,X]]=[[F,X],Y]+[X,[F,Y]]=[F,[X,Y]]\in [F,\cL_2],
\]
and 
\[
\Bigl[F,[X,[F,Y]]+[Y,[F,X]]\Bigr]=[[F,X],[F,Y]]+[[F,Y],[F,X]]=0,
\]
because $[F,[F,\cL_1]]\in\cL_{-3}=0$. This means that $[X,[F,Y]]+[Y,[F,X]]$ belongs to 
$\Cent_{\cL}(\frsl_2)$. The assertion in (v) follows at once from (iii) because 
$[X,[F,Y]]=\frac12([X,[F,Y]]+[Y,[F,X]])+\frac12([X,[F,Y]]-[Y,[F,X]])$.
\end{proof}

\begin{proposition}\label{pr:5graded}
Let $\cL$ be a $5$-graded Lie algebra: $\cL=\cL_{-2}\oplus\cL_{-1}\oplus\cL_0\oplus\cL_1\oplus\cL_2$,
 and assume that there are elements $E\in\cL_2$ and 
$F\in\cL_{-2}$, such that $\espan{E,F,H=[E,F]}$ is a subalgebra isomorphic to $\frsl_2$, with 
$\cL_i=\{X\in\cL\mid [H,X]=iX\}$, for $i=-2,-1,0,1,2$.   Then $\cL_1$ is a $J$-ternary algebra with the triple product
\[
\ptriple{X,Y,Z}=\frac{1}{2}[[X,[F,Y]],Z]
\]
for $X,Y,Z\in\cL_1$.
\end{proposition}
\begin{proof}
$\cL$ splits, as a module for the Lie subalgebra $\frsl_2=\espan{E,F,H}$, as a direct sum of copies
of the adjoint module, the natural module, and the trivial module, as shown in Lemma \ref{le:5graded}.
 Actually, for any  $X\in\cL_2$, 
$\espan{X,[F,X],[F,[F,X]]}$ is (isomorphic to) the adjoint module, while for any $X\in\cL_1$, 
$\espan{X,[F,X]}$ is (isomorphic to) the natural module. The sum of the trivial modules, that is, the 
centralizer of $\frsl_2$, is contained in $\cL_0$. 
Therefore, $\cL$ is endowed with a short 
$\SLs_2$-structure and can be written as in \eqref{eq:shortA1_isotypic}, with $V=\FF p\oplus\FF q$, 
$(p\mid q)=1$, $H,E,F\in\frsl(V)$ with coordinate matrices 
$\left(\begin{smallmatrix} 1&0\\ 0&-1\end{smallmatrix}\right)$,
$\left(\begin{smallmatrix} 0&1\\ 0&0\end{smallmatrix}\right)$, and
$\left(\begin{smallmatrix} 0&0\\ 1&0\end{smallmatrix}\right)$, respectively, 
 in the basis $\{p,q\}$. In particular $\cT$ can be identified with $\cL_1=p\otimes\cT$.

Now, $\cT$ is a $J$-ternary algebra by Theorem \ref{th:JTernarySL2}, with triple product
in \eqref{eq:xyzdxy}, and for $X=p\otimes x$, $Y=p\otimes y$, and $Z=p\otimes z$ in $\cL_1$, we use
 \eqref{eq:shortA1bracket} to get
\[
\begin{split}
[[X,[F,Y]],Z]&=[[p\otimes x,[F,p\otimes y]],p\otimes z]\\
  &=[[p\otimes x,q\otimes y],p\otimes z]\\
&=[\gamma_{p,q} \otimes \langle x\mid y\rangle + d_{x,y},p\otimes z]\\
&=\gamma_{p,q}(p)\otimes\langle x\mid y\rangle\bullet z + p\otimes d_{x,y}(z)\\
&=p\otimes\Bigl(d_{x,y}(z)-\langle x\mid y\rangle\bullet z\Bigr)\\
&=p\otimes 2\ptriple{x,y,z},
\end{split}
\]
thus proving our result.
\end{proof}

\begin{theorem}\label{th:structurable_Jternary}
Let $(\cA,-)$ be a structurable algebra and let $s\in\cS$ be a skew-symmetric element such that the left
multiplication $L_s$ is bijective. Then $\cA$ is a $J$-ternary algebra with the triple product
\[
\ptriple{x,y,z}=V_{x,sy}(z)
\]
for $x,y,z\in\cA$.

Moreover, $\cS$ is a Jordan algebra with multiplication
\[
a\cdot b=\frac12\bigl(a(sb)+b(sa)\bigr),
\]
for $a,b\in\cS$;
$\cA$ is a special Jordan $\cS$-module with
\[
a\bullet x=a(sx),
\]
for $a\in\cS$ and $x\in\cA$; and $\cA$ is an $\cS$-ternary algebra with the triple product above and with
\[
\langle x\mid y\rangle=y\overline{x}-x\overline{y},
\]
for $x,y\in\cA$.

Moreover, if $t$ is the element in $\cA$ with $L_s(t)=1$, then $t$ is in $\cS$ and the Lie algebra 
$S(\cA,\cA)\leq\End_\FF(\cA)^{(-)}$ in \eqref{eq:STT} is given by
\[
S(\cA,\cA)=\{T\in\mathfrak{instrl}(\cA,-)\mid T^\delta(t)=0\}.
\]
\end{theorem}
\begin{proof}
Since $L_s$ is invertible, if $t\in\cA$ is the element such that $st=1$ then, by skew-alternativity, 
we have
$s(ts)=(st)s=s=s1$, which gives $L_s(ts-1)=0$ and $ts=1$. Taking conjugates we get 
$-s\overline{t}=1$, so $L_s(t+\overline{t})=0$ and $t\in\cS$. Moreover, by skew-alternativity, we
have the Moufang identity (proof as in \cite[\S 2.1]{Schafer}) $s(t(sx)=(sts)x=sx$ for any $x\in\cA$. 
Thus $L_sL_tL_s=L_s$ and the invertibility of $L_s$ gives $L_sL_t=L_tL_s=\id$.

Now, the elements $E=(0,t)$ and $F=(0,s)^\sim$ in $\cK(\cA,-)$ (Proposition \ref{pr:KantorLieAlgebra})
satisfy the hypotheses of Proposition \ref{pr:5graded}, which imply that 
$\cK(\cA,-)_1=\{(x,0)\mid x\in\cA\}$ 
is a $J$-ternary
algebra with the triple product
\begin{multline*}
\ptriple{(x,0),(y,0),(z,0)}=\frac12 [[(x,0),[F,(y,0)]],(z,0)]=\frac12 [[(x,0),(sy,0)^\sim],(z,0)]\\
     =\frac12 [V_{x,sy},(z,0)]=\frac12 (V_{x,sy}(z),0),
\end{multline*}
for $x,y,z\in\cA$, and hence we obtain that $\cA$ is a $J$-ternary algebra with the given triple product
 after scaling  by $2$. (Alternatively, 
the version of $\cK(\cA,-)$ in Remark \ref{re:KA-2} can be used to avoid the appearance of $\frac12$ in
the formula above.)

Let us consider now the operators in \eqref{eq:LK} for this $J$-ternary algebra (i.e., special $(1,1)$ 
Freudenthal-Kantor triple system), where
$\ptriple{x,y,z}=V_{x,sy}(z)$. To distinguish it from the $K$-operators $K_{a,b}$ in 
\eqref{eq:Kab_str} we will
denote them by $K^s$ remarking its dependence on the fixed element $s\in\cS$. For $x,y,z\in\cA$,
we get
\begin{multline*}
K^s(x,y)(z)=\ptriple{x,z,y}-\ptriple{y,z,x}
  = V_{x,sz}(y)-V_{y,sz}(x)=K_{x,y}(sz)\\
  =L_{x\overline{y}-y\overline{x}}(sz)=L_{x\overline{y}-y\overline{x}}L_s(z),
\end{multline*}
so that $K^s(\cA,\cA)=L_{\cS}L_s$. As $L_tL_s=\id$, the Jordan algebra $\cJ=\FF\id+K^s(\cA,\cA)$
is just $\cJ=K^s(\cA,\cA)=L_{\cS}L_s$. 

For any $a\in\cS$, by skew-alternativity we have, as before, the Moufang identity $L_aL_sL_a=L_{a(sa)}$, 
and hence, for any $a,b\in\cS$, $L_aL_sL_b+L_bL_sL_a=L_{a(sb)+b(sa)}$, so that the Jordan product 
in $\cJ=L_{\cS}L_s$ becomes
\[
\frac12\bigl((L_aL_s)(L_bL_s)+(L_bL_s)(L_aL_s)\bigr)=\frac12L_{a(sb)+b(sa)}L_s,
\]
for $a,b\in\cS$. Transferring this Jordan product to $\cS$ through the linear bijection 
$\cS\rightarrow \cJ$, $a\mapsto L_aL_s$, shows that $\cS$ is a Jordan algebra with the product
\[
a\cdot b=\frac{1}{2}\bigl(a(sb)+b(sa)\bigr),
\] 
and that $\cA$ is a special module for $\cS$ with $a\bullet x=L_aL_s(x)=a(sx)$,
for $a\in\cS$ and $x\in \cA$.

Theorem \ref{th:JTHein_Allison} shows that $\cA$ is an $\cS$-ternary algebra with the
given triple product and with
\[
\langle x\mid y\rangle =y\overline{x}-x\overline{y}
\]
for $x,y\in\cA$, because 
$-K^s(x,y)=-L_{x\overline{y}-y\overline{x}}L_s=L_{y\overline{x}-x\overline{y}}L_s$, which corresponds,
through the linear bijection above, to the element $y\overline{x}-x\overline{y}\in\cS$.

Finally, from the definition of the Lie algebra $S(\cA,\cA)$ we get
\[
\begin{split}
S(\cA,\cA)&=\espan{V_{x,sy}+V_{y,sx}\mid x,y\in\cA}\\
 &=\espan{[(x,0),[(0,s)^\sim,(y,0)]]+[(y,0),[(0,s)^\sim,(x,0)]]\mid x,y\in\cA}\\
 &=\Cent_{\cK(\cA,-)}(\frsl_2)\quad \text{($\frsl_2=\espan{(0,t),\id,(0,s)^\sim}$)}\\
 &=\{T\in\mathfrak{instrl}(\cA,-)\mid [(0,t),T]=0\}\\
 &=\{T\in\mathfrak{instrl}(\cA,-)\mid T^\delta(t)=0\},
\end{split}
\]
where we have used items (v) and (ii) of Lemma \ref{le:5graded}. (Note that 
$\cK(\cA,-)_0=\mathfrak{instrl}(\cA,-)=[\cK(\cA,-)_1,\cK(\cA,-)_{-1}]$ and
$\cK(\cA,-)_2=[\cK(\cA,-)_1,\cK(\cA,-)_1]$.)
\end{proof}

\smallskip

\begin{example}\label{ex:hermitian_prototypical}
Let $(\cA=\cE\oplus\cW,-)$ be a structurable algebra of a hermitian form 
(see \cite[Example 6.5]{AllisonFaulkner_Steinberg}), that is,
\begin{itemize}
\item 
$(\cE,-)$ is an associative algebra with involution, with multiplication denoted by juxtaposition,

\item 
$\cW$ is a left $\cE$-module with action denoted by $e\circ x$ for $e\in\cE$ and $x\in\cW$,

\item 
$\cW$ is endowed with a hermitian form $\hup:\cW\times\cW\rightarrow \cE$, so that 
$\hup(y,x)=\overline{\hup(x,y)}$ and $\hup(e\circ x,y)=e\hup(x,y)$, for any $x,y\in\cW$ and $e\in\cE$,

\item 
the involution on $\cA$ is defined by $\overline{e+x}=\overline{e}+x$, for $e\in\cE$ and $x\in\cW$, 
that is, the involution on $\cE$ is extended to an order $2$ linear automorphism of $\cA$ by
imposing that $\cW$ consists of symmetric elements.

\item
The multiplication in $\cA$ is defined as follows:
\[
(e_1+x_1)(e_2+x_2)=\bigl(e_1e_2+\hup(x_2,x_1)\bigr)+\bigl(\overline{e_1}\circ x_2+e_2\circ x_1\bigr)
\]
for $e_1,e_2\in\cE$ and $x_1,x_2\in\cW$.
\end{itemize}
Take an element $s\in\cS=\{x\in\cA\mid \overline{x}=-x\}=\{a\in\cE\mid\overline{a}=-a\}$, 
with $L_s$ invertible. Theorem \ref{th:structurable_Jternary} shows that $\cS$ is a Jordan algebra
with $a\cdot b=\frac12 (asb+bsa)$ for $a,b\in\cS$ ($\cE$ is associative, so there is no need 
of parentheses), and that $\cA$ is an $\cS$-ternary algebra with 
\[
a\bullet x=a(sx),\quad \langle x\mid y\rangle =y\overline{x}-x\overline{y},\quad
\ptriple{x,y,z}=V_{x,sy}(z),
\]
for $a\in\cS$, $x,y,z\in\cA$.

Let us show that this ternary triple product is prototypical (Example \ref{ex:prototypical}).

To begin with, consider the \emph{isotope} $\cE^{(s)}$, which is the associative algebra defined on
$\cE$ with new multiplication $a*b\bydef asb$ for any $a,b\in\cE$, and with involution 
$\tau(a)=-\overline{a}$ for any $a\in\cE$. Note that $\tau$ is indeed an involution because 
$\overline{s}=-s$. The unit element of $\cE^{(s)}$ is $t=s^{-1}$, and the Jordan algebra of
symmetric elements is $H(\cE^{(s)},\tau)=\{a\in\cE\mid \tau(a)=a\}=\cS$, with the product
$\frac12 (a*b+b*a)=\frac12(asb+bsa)=a\cdot b$ for $a,b\in\cS$.

Moreover, $\cW$ is a left $\cE^{(s)}$-module with 
\[
a* x\bydef (as)\circ x=a\circ(s\circ x)
\]
for any $a\in \cE$ and $x\in\cW$. Hence, for any $a\in\cS$, $x\in\cW$, and $e\in\cE$,
$a\bullet x=a(sx)=\overline{a}\circ(\overline{s}\circ x)=a\circ(s\circ x)=a* x$, while 
$a\bullet e=ase=a*e$. We conclude that $\cA$ is a left $\cE^{(s)}$-module with action given by
$a\bullet x$ for $a\in\cE$ and $x\in\cA$.

For $x,y\in\cW$ and $a,b\in\cE$, we have
\[
\begin{split}
\langle x\mid y\rangle &=y\overline{x}-x\overline{y}=yx-xy=\hup(x,y)-\hup(y,x),\\
\langle a\mid x\rangle &=x\overline{a}-a\overline{x}=\overline{a}\circ x-\overline{a}\circ x=0,\\
\langle a\mid b\rangle &=b\overline{a}-a\overline{b}.
\end{split}
\]
Define a bilinear form 
\[
\tilde \hup:\cA\times\cA\rightarrow \cE^{(s)}
\]
by imposing $\tilde\hup(\cE,\cW)=\tilde\hup(\cW,\cE)=0$,
$\tilde\hup(x,y)=\hup(x,y)$ for $x,y\in\cW$, and $\tilde \hup(a,b)=-a\overline{b}$ for $a,b\in\cE$. 
By its own definition we get
\[
\langle x\mid y\rangle =\tilde \hup(x,y)-\tilde \hup(y,x),
\]
for any $x,y\in\cA$. For $x,y\in\cW$, $e,f\in\cE$, and $a\in\cS$, we have:
\[
\begin{split}
\tilde \hup(y,x)&=\hup(y,x)=\overline{\hup(x,y)}=-\tau\bigl(\tilde \hup(x,y)\bigr),\\
\tilde \hup(f,e)&=-f\overline{e}=-\overline{e\overline{f}}
      =\tau\bigl(e\overline{f}\bigr)=-\tau\bigl(\tilde\hup(e,f)\bigr),\\
\tilde \hup(a\bullet x,y)&=\hup((as)\circ x,y)=(as)\hup(x,y)=a\bullet \hup(x,y)=a\bullet\tilde \hup(x,y),\\
\tilde \hup(a\bullet e,f)&=-ase\overline{f}=a\bullet\tilde \hup(e,f).
\end{split}
\]
This shows that $\tilde \hup$ is a skew-hermitian form. We must finally check that 
$\ptriple{x,y,z}=\tilde \hup(x,y)\bullet z+\tilde \hup(z,x)\bullet y +\tilde \hup(z,y)\bullet x$ for any $x,y,z\in\cA$.
We split this into several cases.

For $x,y,z\in\cW$ we get
\[
\begin{split}
\ptriple{x,y,z}&=V_{x,sy}(z)
      =\bigl(x(\overline{sy})\bigr)z+\bigl(z(\overline{sy})\bigr)x-\bigl(z\overline{x}\bigr)(sy)\\
 &=-\bigl(x(ys)\bigr)z-\bigl(z(ys)\bigr)x-(zx)(sy)\\
 &=-\hup(s\circ y,x)z-\hup(s\circ y,z)x-\hup(x,z)(sy)\\
 &=-\hup(x,s\circ y)\circ z-\hup(z,s\circ y)\circ x+\hup(z,x)\circ(s\circ y)\quad 
                        \text{(as $ex=\overline{e}\circ x$)}\\
 &=\bigl(\hup(x,y)s\bigr)\circ z+\bigl(\hup(z,y)s\bigr)\circ x+\hup(z,x)\circ(s\circ y)\\
 &=\tilde\hup(x,y)\bullet z+\tilde\hup(z,y)\bullet x+\tilde\hup(z,x)\bullet y,
\end{split}
\]
as desired.

For $e,f,g\in\cE$ we get 
\[
\begin{split}
\ptriple{e,f,g}&=V_{e,sf}(g)
       =\bigl(e(\overline{sf})\bigr)g+\bigl(g(\overline{sf})\bigr)e-(g\overline{e})(sf)\\
  &=-e\overline{f}sg-g\overline{f}se-g\overline{e}sf\\
  &=\tilde\hup(e,f)\bullet g+\tilde\hup(g,f)\bullet e+\tilde\hup(g,e)\bullet f.
\end{split}
\]

Now, for $x,y\in\cW$ and $e\in\cE$,
\[
\begin{split}
\ptriple{x,e,y}&=V_{x,se}(y)
      =\bigl(x(\overline{se})\bigr)y+\bigl(y(\overline{se})\bigr)x-(y\overline{x})(se)\\
 &=-\bigl((\overline{e}s)\circ x\bigr)y-\bigl((\overline{e}s)\circ y\bigr)x-(yx)(se)\\
 &=-\hup(y,(\overline{e}s)\circ x)-\hup(x,(\overline{e}s)\circ y)-\hup(x,y)se\\
 &=-\hup(y,x)\overline{s}e-\hup(x,y)\overline{s}e-\hup(x,y)se\\
 &=\tilde\hup(y,x)\bullet e,
\end{split}
\]
as required, because $\tilde\hup(\cW,\cE)=0$. We also have
\[
\ptriple{e,x,y}=\ptriple{x,e,y}+\langle e\mid x\rangle\bullet y=\ptriple{x,e,y}=\tilde\hup(y,x)\bullet e
\]
using \eqref{eq:JTAllison4}, and
\[
\ptriple{x,y,e}=\ptriple{e,y,x}-\langle x\mid e\rangle\bullet y=\ptriple{e,y,x}=\ptriple{y,e,x}
 =\tilde\hup(x,y)\bullet e
\]
using \eqref{eq:JTAllison3}, \eqref{eq:JTAllison4}, and $\langle\cW\mid\cE\rangle=0$.

Finally, for $e,f\in\cE$ and $x\in\cW$, we get
\[
\begin{split}
\ptriple{x,e,f}&=V_{x,se}(f)
     =\bigl(x(\overline{se})\bigr)f+\bigl(f(\overline{se})\bigr)x-(f\overline{x})(se)\\
 &=-\bigl((\overline{e}s)\circ x\bigr)f-(f\overline{e}s)x-(se)\circ(fx)\\
 &=-f\circ\bigl((\overline{e}s)\circ x\bigr)+(se\overline{f})\circ  x -(se\overline{f})\circ x\\
 &=-(f\overline{e})\bullet x=\tilde\hup(f,e)\bullet x
\end{split}
\]
as desired, and also, using \eqref{eq:JTAllison4}, \eqref{eq:JTAllison3}, and 
$\langle\cW\mid\cE\rangle=0$ we have
\[
\ptriple{e,x,f}=\ptriple{x,e,f}+\langle e\mid x\rangle \bullet f=\ptriple{x,e,f}=\tilde\hup(f,e)\bullet x
\]
and 
\[
\ptriple{e,f,x}=\ptriple{x,f,e}-\langle e\mid x\rangle\bullet f=\ptriple{x,f,e}=\tilde\hup(e,f)\bullet x.
\]

We conclude that indeed, $\cA$ is an $\cS$-ternary algebra of prototypical type.

Also note that the structurable algebras of hermitian forms include the associative algebras with
involution ($\cW=0$).
\end{example}

\bigskip


\section{From algebras to superalgebras via tensor categories}\label{se:tensor}

This section is devoted to reviewing the process to get Lie superalgebras from Lie algebras in the 
symmetric tensor categories $\Repap$ over fields of prime characteristic $p$. Details may be found
in \cite{Kannan,DES,ElduqueEtingofKannan} and the references there in.

Given a symmetric tensor category $\mathfrak{C}$, an \emph{operadic Lie algebra}  is an object $\frg$
in $\mathfrak{C}$, with a morphism $\beta:\frg\otimes\frg\rightarrow\frg$, such that the following
relations hold:
\begin{gather*}
\beta\circ(\id_{\frg\otimes\frg}+c_{\frg,\frg})=0\quad\text{(anticommutativity)},\\[6pt]
\beta\circ(\beta\otimes\id_\frg)\circ\Bigl(\id_{\frg\otimes\frg\otimes\frg}
+(\id_{\frg}\otimes c_{\frg,\frg})\circ (c_{\frg,\frg}\otimes\id_{\frg})\hspace*{1in}\null\\[-2pt]
\null\hspace*{1in}+(c_{\frg,\frg}\otimes\id_{\frg})\circ(\id_{\frg}\otimes c_{\frg,\frg})\Bigr)=0
\quad\text{(Jacobi identity)},
\end{gather*}
where $c_{X,Y}$ is the braiding $X\otimes Y\rightarrow Y\otimes X$ in $\mathfrak{C}$.

For instance, an operadic Lie algebra in the category of vector superspaces $\sVec$ over our ground field
$\FF$, where the braiding is given by $c(x\otimes y)=(-1)^{\lvert x\rvert\lvert y\rvert}y\otimes x$ for
homogeneous elements $x$ and $y$, where $\lvert x\rvert$ is the parity: $0$ or $1$, of $x$, is
a vector superspace $\frg=\frg\subo\oplus \frg\subuno$, with a degree preserving bilinear map
$\frg\otimes \frg\rightarrow \frg$, $x\otimes y\to [x,y]$, satisfying the conditions
\begin{gather*}
[x,y]=-(-1)^{\lvert x\rvert\lvert y\rvert}[y,x],\\
[[x,y],z]+(-1)^{\lvert x\rvert(\lvert y\rvert+\lvert z\rvert)}[[y,z],x]
+(-1)^{\lvert z\rvert(\lvert x\rvert+\lvert y\rvert)}[[z,x],y] = 0,
\end{gather*}
for homogeneous elements $x,y,z\in\frg$.

\begin{remark}\label{re:weak}
There is a subtle point here, and that is the reason to use the adjective `operadic'. In characteristic $3$,
an operadic Lie algebra in $\sVec$ is also called a \emph{weak Lie superalgebra}. A Lie superalgebra
is a weak Lie superalgebra satisfying the extra condition $[[x,x],x]=0$ for any odd element $x$. (Note
that this does not follow from the Jacobi identity in characteristic $3$, being a stronger condition
than the Jacobi identity for odd elements.) 

Given a weak Lie superalgebra $\frg$, the subspace
\[
\mathfrak{i}=\espan{[[x,x],x]\mid x\in\frg\subuno}
\]
is an ideal of $\frg$, contained in $\frg\subuno$, that satisfies $[\frg,\mathfrak{i}]=0$, due to
Jacobi identity. The quotient $\frg/\mathfrak{i}$ is a Lie superalgebra (with the same even part).
In particular, if $\frg\subuno$ is a sum of irreducible nontrivial modules for $\frg\subo$, then $\frg$ is
a Lie superalgebra.

In the same vein, in characteristic $2$, an operadic Lie algebra in $\sVec$ is a Lie superalgebra if
 it satisfies the extra
condition $[x,x]=0$ for any even element $x$.

If the characteristic is $>3$, an operadic Lie algebra in $\sVec$ is just a Lie superalgebra.
\end{remark}

Assume that the characteristic of the ground field $\FF$ is $p>0$. 
The affine group scheme $\balpha_p$,
looked at as a functor from the category of unital, commutative, associative algebras over $\FF$ to
the category of groups, is the functor that assigns to any $R$ the additive group 
$\balpha_p(R)=\{r\in R\mid r^p=0\}$, and that acts naturally on morphisms. Its representing object is
the Hopf algebra $\FF[X]/(X^p)$. A representation of $\balpha_p$ is nothing else but a vector space
endowed with a nilpotent endomorphism $\delta$ satisfying $\delta^p=0$. 
A Lie algebra in $\Repap$ is just
a Lie algebra $\cL$ over $\FF$ endowed with a nilpotent derivation $\delta\in\Der(\cL)$ with 
$\delta^p=0$.

There are, up to isomorphism, $p$ indecomposable objects in $\Repap$: $L_1,\ldots,L_p$, where 
$L_i$  denotes the indecomposable object of dimension $i$, $1\leq i\leq p$.

A morphism $f:X\rightarrow Y$ in $\Repap$ is said to be \emph{negligible} if $\trace(f\circ g)=0$ for any 
morphism $g:Y\rightarrow X$. The symmetric tensor category $\Repap$ is not semisimple. Its 
semisimplification, which is the
category with the same objects but with morphisms the classes of morphisms in $\Repap$ modulo
the subspace of negligible ones, is the Verlinde category $\Ver_p$. In $\Ver_p$, the indecomposable
objects $L_1,\ldots,L_{p-1}$ become irreducible, while $L_p$ is a zero object.

Assuming $p>2$, the full tensor subcategory of $\Ver_p$ generated by $L_1$ and $L_{p-1}$ is equivalent
to the symmetric tensor category $\sVec$ of vector superspaces (over our ground field). 
Therefore, given a Lie algebra $\cL$ over $\FF$ endowed with a nilpotent derivation 
$\delta\in\Der(\cL)$ such
that $\delta^p=0$, that is, given a Lie algebra $\cL$ in $\Repap$, we obtain, by considering 
the Lie bracket
modulo negligible homomorphisms, an operadic Lie algebra in $\Ver_p$. Moreover, this operadic Lie algebra
in $\Ver_p$ contains a subalgebra, which lies in the subcategory above, and hence it is an operadic Lie algebra  
in $\sVec$. Note that if $p=3$, the tensor categories $\Ver_3$ and $\sVec$ are equivalent,
because $L_3$ is a zero object in $\Ver_3$.

Let us make these comments more precise.

\begin{theorem}\label{th:Lss}
Let $\cL$ be a finite-dimensional Lie algebra over a field $\FF$ of characteristic $p>2$, endowed with
a nilpotent derivation $\delta\in\Der(\cL)$ such that $\delta^p=0$. 
Decompose $\cL$ in a sum of indecomposable modules for the action of $\delta$, so that 
\[
\cL=\cL_1\oplus\cdots \oplus\cL_{p},
\]
where $\cL_i$ is a sum of indecomposable modules of dimension $i$ for the action of $\delta$, 
$1\leq i\leq p$. Write $\cL\subo=\cL_1$ and choose a subspace $\cL\subuno$ of $\cL_{p-1}$ 
complementing $\delta(\cL_{p-1})$: $\cL_{p-1}=\cL\subuno\oplus\delta(\cL_{p-1})$.

Define a bracket on $\Lss\bydef\cL\subo\oplus\cL\subuno$ as follows:
\begin{equation}\label{eq:recipe}
\begin{split}
[x\subo, y\subo] &=\mathrm{proj}_{\cL\subo} ([x\subo, y\subo])\\
[x\subo, y\subuno]&=\mathrm{proj}_{\cL\subuno} ([x\subo, y\subuno])\\
[x\subuno, y\subo]&=\mathrm{proj}_{\cL\subuno} ([x\subuno, y\subo])\\
[x\subuno, y\subuno]
    &=\mathrm{proj}_{\cL\subo} \bigl([x\subuno,\delta^{p-2}(y\subuno)]\bigr)
\end{split}
\end{equation}
for all $x\subo,y\subo\in \cL\subo$ and $x\subuno,y\subuno\in \cL\subuno$, where the projections 
are taken with respect to the decomposition
\begin{equation}\label{eq:L1Lp}
\cL=\cL\subo\oplus\cL_2\oplus\cdots\oplus\cL_{p-2}\oplus\cL\subuno\oplus\delta(\cL_{p-1})\oplus\cL_p.
\end{equation}
Then $\Lss$, with this bracket, is an operadic Lie algebra in $\sVec$, which is isomorphic
to a subalgebra of the operadic Lie algebra in $\Ver_p$ obtained from the Lie algebra
$\cL$ in $\Repap$.
\end{theorem}
\begin{proof}
This is proved in \cite[Recipe 2.8, Corollary 2.9, Remark 2.11]{ElduqueEtingofKannan}.
\end{proof}

\smallskip

We are interested in the following consequence of Theorem \ref{th:Lss} for $J$-ternary algebras.

\begin{corollary}\label{co:ternary_Lie_super}
Let $\cT$ be a $J$-ternary algebra over a field $\FF$ of characteristic $3$.Then the $\ZZ/2\ZZ$-graded
vector space $\Lss(\cT)$ with 
\[
\Lss(\cT)\subo= S(\cT,\cT),\quad \Lss(\cT)\subuno=\cT,
\]
with $S(\cT,\cT)$ in \eqref{eq:STT}, is a weak Lie superalgebra with the bracket given by the usual
bracket in $S(\cT,\cT)\leq\End_\FF(\cT)$, by $[d,x]=d(x)$ for $d\in S(\cT,\cT)$ and $x\in \cT$, and by
\[
[x,y]=S(x,y)
\]
for $x,y\in\cT=\Lss(\cT)\subuno$.
\end{corollary}
\begin{proof}
Fix a symplectic basis $\{p,q\}$ in the two-dimensional vector space $V$ in Theorem \ref{th:JTernarySL2}, that is, $(p\mid q)=1$;  take the endomorphism $F\in\frsl(V)$ with coordinate matrix
$\left(\begin{smallmatrix}0&0\\ 1&0\end{smallmatrix}\right)$ in this basis (i.e., $p\to q\to 0$)  and consider the Lie algebra $\cL(\cT)$ in \eqref{eq:L(T)}. Then $\delta=\ad_{F\otimes \id}$ is a nilpotent
derivation of $\cL(\cT)$ with $\delta^3=0$. The summand $\frsl(V)\otimes\cJ$ in \eqref{eq:L(T)} is a sum
of indecomposable modules of dimension $3$ for the action of $\delta$, the summand $V\otimes\cT$ is
a sum of indecomposable modules of dimension $2$, while $S(\cT,\cT)$ is annihilated by $\delta$, and 
hence it is a sum of trivial (indecomposable) modules of dimension $1$. Besides, 
$V\otimes\cT=(p\otimes \cT)+\delta(V\otimes\cT)$, so that we get a decomposition as in
\eqref{eq:L1Lp}, with 
\[
\cL\subo=S(\cT,\cT),\quad \cL\subuno=p\otimes\cT.
\] 
Note that $S(\cT,\cT)$ is a subalgebra of $\cL(\cT)$, and for $d\in S(\cT,\cT)$ and $x\in \cT$, 
$[d,p\otimes x]=p\otimes d(x)$, by \eqref{eq:L(T)bracket}, which also gives
\[
\mathrm{proj}_{\cL\subo} \bigl([p\otimes x,\delta(p\otimes y)]\bigr)
=\mathrm{proj}_{S(\cT,\cT)}[p\otimes x,q\otimes y]=S(x,y),
\]
for any $x,y\in\cT$.
The result then follows at once from Theorem \ref{th:Lss}, by identifying $p\otimes\cT$ with $\cT$ by
means of $p\otimes x\leftrightarrow x$ for $x\in\cT$.
\end{proof}

\smallskip

\begin{remark}
Corollary \ref{co:ternary_Lie_super} can be proved without using the semisimplification process, 
but it has been this process the one that has allowed to realize that it was possible to
define $\Lss(\cT)$ and to check that this is a weak Lie superalgebra.

On the other hand, $\Lss(\cT)$ may fail to be a bona fide Lie superalgebra. For example, let $\cT$ be
a two-dimensional vector space over a field $\FF$ of characteristic $3$ with basis $\{x,y\}$ and
triple product determined by $xxx=y$ and $uvw=0$ if at least one of $u$, $v$, or $w$ equals $y$.
With this triple product, $\cT$ becomes trivially a $J$-ternary algebra. Besides, $S(x,x)$ takes $x$ to 
$2y=-y$ and hence we have $[[x,x],x]=S(x,x)x=-y\neq 0$ in $\Lss(\cT)$. Therefore, the weak
Lie superalgebra $\Lss(\cT)$ is not a Lie superalgebra.
\end{remark}

\bigskip

\section{From \texorpdfstring{$J$}{J}-ternary algebras to Lie superalgebras in characteristic 
\texorpdfstring{$3$}{3}}\label{se:JternarySuper}

Corollary \ref{co:ternary_Lie_super} shows us how to get Lie superalgebras from $J$-ternary algebras
over fields of characteristic $3$. This section will give examples of this situation.

It must be remarked that there is no known classification of the simple finite-dimensional 
$J$-ternary algebras over fields of characteristic $3$. 
The known classification \cite{Hein_degree>2}, \cite{Hein_degree2}, 
needs characteristic $\neq 2,3$. However, all the simple algebras that appear in these
classifications are either of prototypical type or they are obtained from a structurable algebra
with a skew-symmetric element with invertible left multiplication as in 
Theorem \ref{th:structurable_Jternary}. Again, there is no known classification of the simple
finite-dimensional structurable algebras over fields of characteristic $3$, but the known simple
algebras in other characteristics make sense in characteristic $3$ (see 
\cite[\S 6]{AllisonFaulkner_Steinberg}).

\subsection{Lie superalgebras from prototypical \texorpdfstring{$J$}{J}-ternary algebras}\label{ss:proto}

Let us start with the Lie superalgebras obtained from simple $J$-ternary algebras fitting in
 the prototypical example (Example \ref{ex:prototypical}).

\begin{theorem}\label{th:proto_super}
Let $(\cA,-)$ be a finite-dimensional simple algebra with involution (i.e., there is no proper ideal closed
under the involution!) over an algebraically closed field $\FF$ of characteristic $3$. Let $\cT$ be a left
$\cA$-module and let $\hup:\cT\times\cT\rightarrow\cA$ be a nondegenerate skew-hermitian form. Then
$\cT$ is a $J$-ternary algebra with triple product
\[
\ptriple{x,y,z}=\hup(x,y)z+\hup(z,x)y+\hup(z,y)x
\]
for $x,y,z\in\cT$. Let $\Lss(\cT)$ be the associated weak Lie superalgebra in 
Corollary \ref{co:ternary_Lie_super}. Then $\Lss(\cT)$ is a Lie superalgebra and the following conditions 
hold:
\begin{itemize}
\item
If the involution $-$ on $\cA$ is of the first kind, then there are $n,m\in\ZZ_{> 0}$, with $m$ even,
 such that $\Lss(\cT)$ is isomorphic to the orthosymplectic Lie superalgebra $\frosp(n\vert m)$.
\item 
If the involution $-$ on $\cA$ is of the second kind, then there are $n,m\in\ZZ_{> 0}$ such
that $\Lss(\cT)$ is isomorphic to the projective special linear Lie superalgebra $\frpsl(n\vert m)$.
\end{itemize}
\end{theorem}

\begin{proof}

The proof will be done in several steps.

Assume first that the involution is of the first kind and orthogonal. 
Then the algebra $\cA$ is simple and there exists
a finite-dimensional vector space $X$ over $\FF$, endowed with a nondegenerate symmetric bilinear form
$\bup_X:X\times X\rightarrow\FF$, such that $(\cA,-)$ is isomorphic to $\left(\End_\FF(X),-\right)$, with the
involution given by the adjunction relative to $\bup_X$:
\[
\bup_X\bigl(f(u),v\bigr)=b_X(u,\overline{f}(v)\bigr)
\]
for any $u,v\in X$ and $f\in\End_\FF(X)$. We may then assume that $\cA=\End_\FF(X)$ with this
involution. 

Any finite-dimensional simple left module is, up to isomorphism, of the form $\cT=X\otimes Y$ for a 
finite-dimensional vector space $Y$, where the action of $\cA$ is on $X$. 
Let $\hup:\cT\times\cT\rightarrow\cA$ be a 
nondegenerate skew-hermitian form. The left $\cA$ module $\cT$ is also a right module with 
$ta\bydef \overline{a}t$ for any $a\in\cA$ and $t\in\cT$. Then $\hup$ may be seen as a homomorphism
of $\cA$-bimodules $\hup:\cT\otimes\cT\rightarrow\cA$. Up to scalars, the unique homomorphism
of $\cA$-bimodules $X\otimes X\rightarrow \cA=\End_\FF(X)$ is given by 
$x_1\otimes x_2\mapsto x_1 \bup_X(x_2,\cdot)$, for $x_1,x_2\in X$. As a consequence, there exists a
bilinear form $\bup_Y:Y\times Y\rightarrow \FF$ such that
\[
\hup(x_1\otimes y_1,x_2\otimes y_2)=\bup_Y(y_1,y_2)x_1\bup_X(x_2,\cdot),
\]
for any $x_1,x_2\in X$ and $y_1,y_2\in Y$. The fact that $\hup$ is nondegenerate and skew-symmetric
implies that $\bup_Y$ is a nondegenerate skew-symmetric bilinear form. Let us compute the action
of the operators 
\[
S(u,v)=L(u,v)+L(v,u)=\ptriple{u,v,\cdot}+\ptriple{v,u,\cdot}\in\End_\FF(\cT),
\] 
for $u_1=x_1\otimes y_1$ and $u_2=x_2\otimes y_2$ as above:
\[
\begin{split}
S(x_1\otimes y_1&,x_2\otimes y_2)(x\otimes y)\\
  &=\bigl(\hup(x_1\otimes y_1,x_2\otimes y_2)+\hup(x_2\otimes y_2,x_1\otimes y_1)\bigr)(x\otimes y) \\
   &\qquad\qquad          +2\bigl(\hup(x\otimes y,x_1\otimes y_1)(x_2\otimes y_2)
                  +\hup(x\otimes y,x_2\otimes y_2)(x_1\otimes y_1)\bigr)\\
 &=\bigl(\hup(x_1\otimes y_1,x_2\otimes y_2)+\hup(x_2\otimes y_2,x_1\otimes y_1)\bigr)(x\otimes y) \\
     &\qquad\qquad         -\bigl(\hup(x\otimes y,x_1\otimes y_1)(x_2\otimes y_2)
                  +\hup(x\otimes y,x_2\otimes y_2)(x_1\otimes y_1)\bigr)\\
 &=x_1\bup_X(x_2,x)\otimes \bup_Y(y_1,y_2)y + x_2\bup_X(x_1,x)\otimes \bup_Y(y_2,y_1)y\\
 &\qquad\qquad
    -x\bup_X(x_1,x_2)\otimes\bup_Y(y,y_1)y_2 - x\bup_X(x_2,x_1)\otimes\bup_Y(y,y_2)y_1\\
&=\bigl(x_1\bup_X(x_2,x)-x_2\bup_X(x_1,x)\bigr)\otimes \bup_Y(y_1,y_2)y\\
 &\qquad\qquad +\bup_X(x_1,x_2)x\otimes\bigl(\bup_Y(y_1,y)y_2+\bup_Y(y_2,y)y_1\bigr)
\end{split}
\]
for $x\in X$ and $y\in Y$. Hence, we have
\begin{equation}\label{eq:Sx1y1x2y2}
S(x_1\otimes y_1,x_2\otimes y_2)=\sigma_{x_1,x_2}\otimes\bup_Y(y_1,y_2)\id_Y
  +\bup_X(x_1,x_2)\id_X\otimes \gamma_{y_1,y_2}
\end{equation}
where 
\[
\sigma_{x_1,x_2}=x_1\bup_X(x_2,\cdot)-x_2\bup_X(x_1,\cdot)\quad \text{and}\quad 
\gamma_{y_1,y_2}=y_1\bup_Y(y_2,\cdot)+y_2\bup_Y(y_1,\cdot).
\] 
(Note that the operators
$\sigma_{x_1,x_2}$ span the orthogonal Lie algebra of $(X,\bup_X)$, and the operators
$\gamma_{y_1,y_2}$ span the symplectic Lie algebra of $(Y,\bup_Y)$.)

Consider the vector superspace with even part $X$, odd part $Y$, and endowed with the supersymmetric
even bilinear form $\boldbup$ such that 
\[
\boldbup\vert_X=\bup_X\quad \text{and}\quad \boldbup\vert_Y=-\bup_Y.
\]
The corresponding orthosymplectic Lie superalgebra is given by
\[
\frosp(X\vert Y,\boldbup)_i=\{\varphi\in\frgl(X\vert Y)_i\mid 
\boldbup(\varphi(u),v)+(-1)^{i\lvert u\rvert}\boldbup(u,\varphi(v))=0\ \forall u,v\in X\cup Y\}
\]
for $i=\overline{0},\overline{1}$, where $\lvert u\rvert$ is the parity of the homogeneous element $u$.
This is spanned by the operators
\[
\Lambda_{u,v}=u\boldbup(v,\cdot)-(-1)^{\lvert u\rvert\lvert v\rvert}v\boldbup(u,\cdot),
\]
for $u,v\in X\cup Y$. The odd part of 
$\frosp(X\vert Y,\boldbup)$ is $\Lambda_{X,Y}$ which can be identified with $\cT=X\otimes Y$ by 
means of $\Lambda_{x,y}\leftrightarrow x\otimes y$. An easy computation gives
\[
\begin{split}
[\Lambda_{x_1,y_1},\Lambda_{x_2, y_2}]
 &=\left(x_1\boldbup(y_1,\cdot)-y_1\boldbup(x_1,\cdot)\right)\circ
      \left(x_2\boldbup(y_2,\cdot)-y_2\boldbup(x_2,\cdot)\right)\\
 &\qquad\qquad +\left(x_2\boldbup(y_2,\cdot)-y_2\boldbup(x_2,\cdot)\right)\circ
       \left(x_1\boldbup(y_1,\cdot)-y_1\boldbup(x_1,\cdot)\right)\\
 &= - x_1\boldbup(y_1,y_2)\boldbup(x_2,\cdot)-y_1\boldbup(x_1,x_2)\boldbup(y_2,\cdot)\\
 &\qquad\qquad - x_2\boldbup(y_2,y_1)\boldbup(x_1,\cdot)-y_2\boldbup(x_2,x_1)\boldbup(y_1,\cdot)\\
 &=\bup_Y(y_1,y_2)\sigma_{x_1,x_2}+\bup_X(x_1,x_2)\gamma_{y_1,y_2}
\end{split}
\]
for $x_1,x_2\in X$, $y_1,y_2\in Y$. Comparing with \eqref{eq:Sx1y1x2y2}, it turns out that, after 
identifying $\cT$ with $\frosp(X\vert Y,\boldbup)\subuno$ as above, we have
\[
S(x_1\otimes y_1,x_2\otimes y_2)
=\ad([\Lambda_{x_1,y_1},\Lambda_{x_2, y_2}])\vert_\cT
\]
and this shows that $S(\cT,\cT)$ coincides with $\ad(\frosp(X\vert Y,\boldbup)\subo)\vert_\cT$ and,
since the action of $\frosp(X\vert Y,\boldbup)\subo$ on $\frosp(X\vert Y,\boldbup)\subuno$ is faithful, 
that $\Lss(\cT)=S(\cT,\cT)\oplus\cT$ is isomorphic to $\frosp(X\vert Y,\boldbup)$.

\smallskip

The situation for a symplectic involution is completely analogous.

\smallskip

Let us consider finally the situation of a second kind involution. In this case, up to isomorphism,
we may assume $\cA=\cB\oplus\cB^{\op}$, where $\cB^{\op}$ is the opposite algebra, and that
the involution is given by $\overline{b_1+b_2^{\op}}=b_2+b_1^{\op}$, where we write $b^{\op}$ for
the element $b\in\cB$ when considered as an element of $\cB^{\op}$. A left $\cA$-module $\cT$ is then,
up to isomorphism, the direct sum of a left $\cB$-module and a left $\cB^{\op}$-module
(i.e., a right $\cB$-module): $\cT=\cM\oplus \cN$.  Denote
by $\rho\colon\cA\rightarrow\End_\FF(\cT)$ the associated representation.
Write $e_1=(1,0)$ and $e_2=(0,1^{\op})$.

If $\hup:\cT\times\cT\rightarrow \cA$ is a skew-hermitian form, then we get
$\hup(\cM,\cM)=\hup(e_1\cT,e_1\cT)=e_1\hup(\cT,\cT)\overline{e_1}\subseteq e_1\cA e_2=0$, and also
$\hup(\cN,\cN)=0$. Moreover, we have $\hup(\cM,\cN)\subseteq e_1\cA e_1=\cB$ and
$\hup(\cN,\cM)\subseteq \overline{\cB}=\cB^{\op}$. In other words, there is a homomorphism of $\cB$-bimodules
$\mu\colon \cM\otimes \cN\rightarrow \cB$ such that $\hup(m,n)=\mu(m\otimes n)\in\cB$ and 
$\hup(n,m)=-\overline{\hup(m,n)}=-\overline{\mu(m\otimes n)}=-\mu(m\otimes n)^{\op}\in\cB^{\op}$.

For $x,y,z\in \cT$, the operator $S(x,y)\in\End_\FF(\cT)$ works as follows:
\[
\begin{split}
S(x,y)(z)&=\ptriple{x,y,z}+\ptriple{y,x,z}\\
 &=\bigl(\hup(x,y)+\hup(y,x)\bigr)z+2\bigl(\hup(z,x)y+\hup(z,y)x\bigr)\\
 &=\bigl(\hup(x,y)-\overline{\hup(x,y)}\bigr)z-\bigl(\hup(z,x)y+\hup(z,y)x\bigr).
\end{split}
\]
For $x,y\in\cM$, $\hup(x,y)=0$ and $\hup(z,x)y,\hup(z,y)x\in\cB^{\op}\cM=0$, and the same
happens for $x,y\in\cN$. Now, for $x\in\cM$ and $y\in\cN$, $\hup(x,y)=\mu(x\otimes y)\in\cB$. 
In other words, we have, for $x\in\cM$ and $y\in\cN$:
\begin{multline}\label{eq:SxyMN}
S(x,y)=\rho\bigl(\hup(x,y)-\hup(x,y)^{\op}\bigr)-\hup(\cdot,x)y-\hup(\cdot, y)x\colon\\
z\mapsto 
\begin{cases}
\mu(x\otimes y)z-\mu(z\otimes y)x,&\text{if $z\in\cM$,}\\
-z\mu(x\otimes y)+y\mu(x\otimes z),&\text{if $z\in\cN$.}
\end{cases}\quad\null
\end{multline}

Now, $\cB$ is a finite-dimensional simple algebra, so we may assume $\cB=\End_\FF(X)$ for a
finite-dimensional vector space $X$. The finite-dimensional left $\cB$-module $\cM$ is, 
up to isomorphism, of
the form $X\otimes Z$ for a vector space $Z$, where the action on $\cB$ is on $X$. In the same vein,
the right $\cB$-module $\cN$ is, up to isomorphism, of the form  $Y\otimes X^*$, where the action
of $\cB$ is on the dual $X^*$, that is, $\varphi b\colon x\mapsto \varphi(bx)$ for 
$\varphi\in X^*$, $b\in\cB$, and $x\in X$. As a bimodule for $\cB$, $X\otimes X^*$ is isomorphic to 
$\cB=\End_\FF(X)$, with $x\otimes\varphi$ identified with the linear map $z\mapsto x\varphi(z)$,
and the homomorphism of $\cB$-bimodules $\mu:\cM\otimes\cN\rightarrow \cB$ is of the form
\[
(x\otimes z) \otimes (y\otimes\varphi)\mapsto \eta(z,y)x\otimes\varphi,
\] 
for $x\in X$, $\varphi\in X^*$,
$y\in Y$ and $z\in Z$, for a bilinear form $\eta:Z\times Y\rightarrow \FF$.

The skew-hermitian form $\hup$ is nondegenerate if and only if so is the bilinear form $\eta$, and then
$\eta$ can be used to identify $Z$ with $Y^*$. 

Summarizing, for involutions of the second kind, we may assume that $\cA=\cB\oplus\cB^{\op}$, the
involution being the exchange involution, with $\cB=\End_\FF(X)$ for a finite-dimensional vector
space $X$, while $\cT=(X\otimes Y^*)\oplus (Y\otimes X^*)$ for another finite-dimensional vector
space $Y$. Moreover, the skew-hermitian form $\hup$ satisfies that $X\otimes Y^*$ and $Y\otimes X^*$
are isotropic for $\hup$, while we have
\[
\hup(x\otimes \omega,y\otimes \varphi)
=\mu\bigl((x\otimes \omega)\otimes(y\otimes \varphi)\bigr)=\omega(y)x\otimes \varphi,
\]
for $x\in X$, $\varphi\in X^*$, $y\in Y$, and $\omega\in Y^*$.

Equation \eqref{eq:SxyMN} takes now the following form:
\begin{equation}\label{eq:Sxomegayfi}
S(x\otimes \omega,y\otimes\varphi)\colon
  \begin{cases}
   x'\otimes\omega'\mapsto \omega(y)\varphi(x')x\otimes\omega'-\omega'(y)\varphi(x)x'\otimes\omega,&\\
   y'\otimes \varphi'\mapsto -\omega(y)\varphi'(x)y'\otimes\varphi+\omega(y')\varphi(x)y\otimes\varphi',&
\end{cases}
\end{equation}
for $x,x'\in X$, $\varphi,\varphi'\in X^*$, $y,y'\in Y$, and $\omega,\omega'\in Y^*$.

Consider the vector superspace with even part $X$ and odd part $Y$, and the corresponding general
linear Lie superalgebra $\frgl(X\vert Y)=\End_\FF(X\oplus Y)$, with its natural $\ZZ/2\ZZ$-grading and
bracket. Use the natural identifications
\[
\begin{split}
\frgl(X\vert Y)\subo&=(X\otimes X^*)\oplus (Y\otimes Y^*),\\
\frgl(X\vert Y)\subuno&=(X\otimes Y^*)\oplus (Y\otimes X^*)\,(=\cT).
\end{split}
\]
Note that $\frgl(X\vert Y)$ is $3$-graded with 
\[
\frgl(X\vert Y)_{-1}=X\otimes Y^*, \quad
\frgl(X\vert Y)_0=(X\otimes X^*)\oplus (Y\otimes Y^*),\quad 
 \frgl(X\vert Y)_1=Y\otimes X^*.
\]
For any $x\in X$, $\varphi\in X^*$, $y\in Y$, and $\omega\in Y^*$, the bracket in 
$\frgl(X\vert Y)$ of the odd elements $x\otimes\omega$ and $y\otimes\varphi$ is
\[
[x\otimes\omega,y\otimes\varphi]
  =(x\otimes\omega)\circ(y\otimes\varphi) + (y\otimes\varphi)\circ(x\otimes\omega)
  =\omega(y)x\otimes\varphi +\varphi(x) y\otimes\omega.
\]
These elements span the even part of the special linear Lie superalgebra $\frsl(X\vert Y)$.

The action of $[x\otimes\omega,y\otimes\varphi]$ on the elements of $\frgl(X\vert Y)\subuno=\cT$ works 
as follows:
\[
\begin{split}
[[x\otimes\omega,y\otimes\varphi],x'\otimes\omega']
 &=[\omega(y)x\otimes\varphi +\varphi(x) y\otimes\omega,x'\otimes\omega']\\
 &=\omega(y)\varphi(x')x\otimes\omega'-\varphi(x)\omega'(y)x'\otimes\omega,\\
[[x\otimes\omega,y\otimes\varphi],y'\otimes\varphi']
 &=[\omega(y)x\otimes\varphi +\varphi(x) y\otimes\omega,y'\otimes\varphi']\\
 &=\varphi(x)\omega(y')y\otimes\varphi'-\omega(y)\varphi'(x)y'\otimes\varphi,
\end{split}
\]
for $x,x'\in X$, $\varphi,\varphi'\in X^*$, $y,y'\in Y$, and $\omega,\omega'\in Y^*$,
so we conclude that 
\[
S(x\otimes\omega,y\otimes\varphi)=\ad([x\otimes\omega,y\otimes\varphi])\vert_\cT,
\]
and this shows that $S(\cT,\cT)$ equals $\ad(\frsl(X\vert Y)\subo)\vert_\cT$, which is isomorphic
to the even part of the projective special linear Lie superalgebra $\frpsl(X\vert Y)$, and that
$\Lss(\cT)=S(\cT,\cT)\oplus\cT$ is isomorphic to the projective special linear Lie superalgebra
$\frpsl(X\vert Y)$, as required.
\end{proof}

\medskip

\subsection{Lie superalgebras from simple structurable algebras}\label{ss:structurable}

Now we will turn our attention to the $J$-ternary algebras obtained from finite-dimensional simple
structurable algebras with a suitable regular skew-symmetric element 
(Theorem \ref{th:structurable_Jternary}). Example \ref{ex:hermitian_prototypical} shows that, if our structurable
algebra is an associative algebra with involution or the structurable algebra of a hermitian form, the
corresponding $J$-ternary algebra is prototypical. If the structurable algebra is a Jordan algebra (with
identity involution), then there are no nonzero skew-symmetric elements.

\smallskip

On the other hand, if $(\cA,-)$ is a simple finite-dimensional structurable algebra with one-dimensional
space of skew-symmetric elements, then $\cS(\cA,-)=\FF s$ for a skew-symmetric element
$s$ with invertible $L_s$ and with $s^2\in\FF 1$. Denote by $\cT$ the associated $J$-ternary
algebra in Theorem \ref{th:structurable_Jternary}. This last theorem
 and \eqref{eq:JTAllison3} show that the operators 
$K(x,y)\colon z\mapsto \ptriple{x,z,y}-\ptriple{y,z,x}$ are scalar multiples of the identity. Hence,
the Jordan algebra
$\cJ=\FF \id+K(\cT,\cT)$ in Theorem \ref{th:JTernarySL2} is just the one-dimensional unital Jordan algebra.
The outcome is that $\cT$ is then a so called \emph{symplectic triple system}, and the associated
Lie superalgebra $\Lss(\cT)$ has been computed in \cite[Theorem 3.2]{Elduque_NewSimple3}.
The Lie superalgebras that appear are the Brown superalgebra $\frbr(2;3)$ of superdimension 
$(10\vert 8)$, i.e., the even part has dimension $10$ and the odd part has dimension $8$
(see \cite{Elduque_Models}), and the Lie superalgebras $\frg(r,6)$, $r=1,2,4,8$, in the ``supermagic
square'' in \cite{CunhaElduque_Extended} (see \cite[Corollary 5.9]{CunhaElduque_Jordan}). The notation
above for these superalgebras follows \cite{BouarroudjGrozmanLeites}.

\smallskip

We are left with two types of simple structurable algebras, assuming the ground field is 
algebraically closed: the tensor product of two unital composition algebras, with their natural involutions,
such that one of the factors is a Cayley algebra (dimension $8$), as otherwise the algebra is associative
and hence the associated $J$-ternary algebra is prototypical, or the $35$-dimensional Smirnov algebra
\cite{Smirnov}.

For the
Smirnov algebra, this may be defined as the subalgebra $T(\cC)$ of the tensor product $\cC\otimes\cC$, 
where $\cC$ is the Cayley algebra over $\FF$, obtained as the kernel of the map $S^2(\cC)\rightarrow \FF$
that takes $x\otimes x$ to $\nup(x)$ (its norm), where $S^2(\cC)$ denotes the subspace of symmetric
tensors (see \cite{AllisonFaulkner_Smirnov}). The skew-symmetric elements of $T(\cC)$ are of the form
$s\otimes 1+1\otimes s$ for $s\in \cS(\cC,-)$ (skew-symmetric in $\cC$). Due to the skew-alternativity, 
for any $r\in\cS(\cC,-)$, the next product does not need parentheses, and an easy computation gives:
\[
(s\otimes 1+1\otimes s)(r\otimes 1+1\otimes r)(s\otimes 1+1\otimes s)
=-2\nup(r,s)(s\otimes 1+1\otimes s),
\]
where $\nup$ denotes the norm of the Cayley algebra $\cC$. If $L_{s\otimes 1+1\otimes s}$ were
invertible, then we would have 
\[
0=(r\otimes 1+1\otimes r)(s\otimes 1+1\otimes s)=rs\otimes 1+r\otimes s+s\otimes r+1\otimes rs
\]
for any $r\in\cS=\cS(\cC,-)$ such that $\nup(r,s)=0$. Since 
$\cC\otimes\cC=(\cC\otimes 1+1\otimes\cC)\oplus (\cS\otimes\cS)$, this would give 
$r\otimes s+s\otimes r=0$ for any $r\in\cS=\cS(\cC,-)$ such that $\nup(r,s)=0$, which is not the case.
Therefore, there are no suitable skew-symmetric elements in $T(\cC)$.

\smallskip

This leaves us with the tensor products of two unital composition algebras, and this case will be dealt
 with in the next section.

\bigskip

\section{A magic square of Lie superalgebras}\label{se:magic}

Unless otherwise stated, \emph{the ground field $\FF$ will be assumed to be algebraically closed
of characteristic $3$ throughout this section}.

Let $\cC_1$ be the Cayley algebra over $\FF$, that is, the (unique) eight-dimensional unital composition
algebra, and let $\cC_2$ be another unital composition algebra. Consider the structurable algebra 
$\cA=\cC_1\otimes \cC_2$, where the involution is the tensor product of the
canonical involutions in  $\cC_1$ and $\cC_2$: $\overline{a\otimes b}=\overline{a}\otimes\overline{b}$,
for any $a\in\cC_1$ and $b\in\cC_2$.

Denote by $\nup_i$ the norm in $\cC_i$, and by $\cS_i$ the subspace of trace zero elements (i.e., 
$\bar s=-s$) in $\cC_i$, for $i=1,2$. The skew-symmetric part of $\cA$ is 
$\cS=\cS_1\otimes 1 + 1\otimes\cS_2$, which we will identify with $\cS_1\oplus\cS_2$. 

As in \cite{Allison_Tensor} consider the \emph{Albert form} $Q:\cS\rightarrow\FF$, which is the 
nondegenerate quadratic form given by 
\begin{equation}\label{eq:AlbertForm}
Q(s_1+s_2)\bydef \nup_1(s_1)-\nup_2(s_2),
\end{equation}
for $s_1\in\cS_1$ and $s_2\in\cS_2$. Consider also the linear map $\sharp$ given by
\[
(s_1+s_2)^\sharp \bydef s_1-s_2.
\]
This is an isometry of $Q$.

The following result is \cite[Proposition 3.3]{Allison_Tensor}. The proof there works in characteristic 
$\neq 2$.

\begin{proposition}\label{pr:AlbertForm}
Let $(\cA,-)$ be the structurable algebra obtained as the tensor product of the Cayley algebra $\cC_1$ 
and a unital composition algebra $\cC_2$, and consider its Albert form $Q:\cS\rightarrow\FF$ in
\eqref{eq:AlbertForm}. Then the following conditions hold for any $a,b,c\in\cS$, where 
$L_x\colon \cA\rightarrow\cA$ denotes, as usual, the left multiplication by an element $x\in\cA$, and 
$Q(a,b)=Q(a+b)-Q(a)-Q(b)$ is the polar form of $Q$:
\begin{gather}
L_aL_{a^\sharp}=L_{a^\sharp}L_a=-Q(a)\id,\label{eq:AllQ1}\\
ab^\sharp a=Q(a)b-Q(a,b)a,\label{eq:AllQ2}\\
L_aL_{b^\sharp}L_a=Q(a)L_b-Q(a,b)L_a,\label{eq:AllQ3}\\
L_aL_{b^\sharp}L_aL_{c^\sharp}=Q(a)L_bL_{c^\sharp}-Q(a,b)L_aL_{c^\sharp},\label{eq:AllQ4}\\
(ab^{\sharp})^2+Q(a,b)ab^{\sharp}+Q(a)Q(b)1=0.\label{eq:AllQ5}
\end{gather}
\end{proposition}

\smallskip

\subsection{The inner structure Lie algebra}

Our first goal is the description of the inner structure Lie algebra of these structurable algebras.

We need the following preliminary result.

\begin{lemma}\label{le:LSLS}
Let $(\cA,-)$ be the structurable algebra obtained as the tensor product of the Cayley algebra $\cC_1$ 
and a unital composition algebra $\cC_2$. Then the subspace 
$L_\cS L_\cS=\espan{L_aL_b\mid a,b\in\cS}$ is an ideal of $\mathfrak{instrl}(\cA,-)$ that splits
as the direct sum of its central ideal $\FF\id$ and the ideal
\[
M_{\cS,\cS}=\espan{M_{a,b}\bydef L_aL_{b^\sharp}-L_bL_{a^\sharp}\mid a,b\in\cS}.
\]
Moreover, the representation $\delta\colon\mathfrak{instrl}(\cA,-)\rightarrow\End_{\FF}(\cS)$ involved
in the definition of $\cK(\cA,-)$ in \eqref{eq:KA-} takes $\id$ to $2\id_{\cS}$ and takes 
$M_{\cS,\cS}$ isomorphically onto the orthogonal Lie algebra $\frso(\cS,Q)$.
\end{lemma}

\begin{proof}
The fact that $L_{\cS}L_{\cS}$ is an ideal of $\mathfrak{instrl}(\cA,-)$ follows from 
 $\mathfrak{instrl}(\cA,-)$ being equal to $\cK(\cA,-)_0$, and $L_{\cS}L_{\cS}$ to
$[\cK(\cA,-)_2,\cK(\cA,-)_{-2}]$ in Proposition \ref{pr:KantorLieAlgebra}. The arguments in \cite[Proposition 4.3]{Allison_Tensor}, which are
valid in characteristic $\neq 2$, give $L_{\cS}L_{\cS}=\FF\id\oplus M_{\cS,\cS}$ and
\begin{equation}\label{eq:Mab_sigma}
M_{a,b}^\delta(c)=2\bigl(Q(a,c)b-Q(b,c)a\bigr)\bydef 2\sigma_{a,b}^Q(c)
\end{equation}
for any $a,b,c\in\cS$. By skew-symmetry, the dimension of $M_{\cS,\cS}$ is at most 
$\binom{\dim\cS}{2}=\dim\frso(\cS,Q)$, and the orthogonal Lie algebra $\frso(\cS,Q)$ is spanned
by the operators $\sigma_{a,b}^Q$, so we obtain that $\delta$ takes $M_{\cS,\cS}$ isomorphically
to $\frso(\cS,Q)$. The result follows.
\end{proof}

\begin{theorem}\label{th:instrl(C1C2)}
Let $(\cA,-)$ be the structurable algebra obtained as the tensor product of the Cayley algebra $\cC_1$ 
and a unital composition algebra $\cC_2$.
\begin{itemize}
\item
If the dimension of $\cC_2$ is $1$, $2$ or $8$, then we have the equality 
$\mathfrak{instrl}(\cA,-)=L_\cS L_\cS$.

\item
In the remaining case: $\dim\cC_2=4$, the Lie algebra $\mathfrak{instrl}(\cA,-)$ is the direct sum
of its ideals $R_{\cS_2}=\espan{R_a\mid a\in\cS_2}$ and $L_\cS L_\cS$. (As usual, $R_a$ denotes
the right multiplication by $a$, and $\cS_2$ is identified with $1\otimes\cS_2$.)
\end{itemize}
\end{theorem}

\begin{remark}
Over fields of characteristic $\neq 2,3$ the above result fails for $\dim\cC_2=2$, where 
$\mathfrak{instrl}(\cA,-)=R_{\cS_2}\oplus L_\cS L_\cS$ (\cite[Theorem 4.4]{Allison_Tensor}).
\end{remark}

\smallskip

The proof of Theorem \ref{th:instrl(C1C2)} will be given in several steps and in a quite indirect way.
In the process, some results that have their own interest will be proved.

\smallskip

Let us recall first the definition in \cite{AllisonFaulkner_Steinberg} of the Steinberg unitary algebra 
$\frstu_3(\cA,-,\gamma)$
attached to a structurable algebra $(\cA,-)$ and a triple 
$(\gamma_1,\gamma_2,\gamma_3)\in(\FF^\times)^3$. This is the Lie algebra with generators
$u_{ij}[x]$, for $1\leq i\neq j\leq 3$ and $x\in\cA$, subject to the following relations for any $x,y\in\cA$
and indices $i\neq j$:
\begin{itemize}
\item $u_{ij}[x]$ is linear in $x$,
\item $u_{ij}[x]=-\gamma_i\gamma_j^{-1}u_{ji}[\overline{x}]$,
\item $[u_{12}[x],u_{23}[y]]=u_{13}[xy]$, and cyclically on $1,2,3$.
\end{itemize}
Let us write explicitly the last condition, taking into account the previous one:
\begin{equation}\label{eq:u12u23}
\begin{split}
[u_{12}[x],u_{23}[y]]&=-\gamma_1\gamma_3^{-1}u_{31}[\overline{xy}],\\
[u_{23}[x],u_{31}[y]]&=-\gamma_2\gamma_1^{-1}u_{12}[\overline{xy}],\\
[u_{31}[x],u_{12}[y]]&=-\gamma_3\gamma_2^{-1}u_{23}[\overline{xy}].
\end{split}
\end{equation}

It follows that $\frstu_3(\cA,-,\gamma)$ is $(\ZZ/2\ZZ)^2$-graded:
\[
\frstu_3(\cA,-,\gamma)=\frg\oplus u_{12}[\cA]\oplus u_{23}[\cA]\oplus u_{31}[\cA],
\]
with $\frstu_3(\cA,-,\gamma)_{(\bar 0,\bar 0)}=
  \frg=[u_{12}[\cA],u_{21}[\cA]]+[u_{23}[\cA],u_{32}[\cA]]+[u_{31}[\cA],u_{13}[\cA]]$,
$\frstu_3(\cA,-,\gamma)_{(\bar 1,\bar 0)}=u_{12}[\cA]\,(=u_{21}[\cA])$, 
$\frstu_3(\cA,-,\gamma)_{(\bar 0,\bar 1)}=u_{23}[\cA]\,(=u_{32}[\cA])$, and
$\frstu_3(\cA,-,\gamma)_{(\bar 1,\bar 1)}=u_{31}[\cA]\,(=u_{13}[\cA])$.

\smallskip

On the other hand, the Lie algebra $\cK(\cA,-)$ in Proposition \ref{pr:KantorLieAlgebra} is endowed
with two natural order $2$ commuting automorphisms (see \cite{EldOku_S4}):
\[
\begin{split}
\epsilon\colon (x,a)^\sim+T+(y,b)&\mapsto (y,b)^\sim+T^\varepsilon+(x,a),\\
\tau\colon  (x,a)^\sim+T+(y,b)&\mapsto (-x,a)^\sim+T+(-y,b).
\end{split}
\]
These two automorphisms induce a grading by $(\ZZ/2\ZZ)^2$:
\[
\begin{split}
\cK(\cA,-)_{(\bar 0,\bar 0)}&=\{T\in \frinstrl(\cA,-)\mid T^\varepsilon=T\}\oplus
       \{(0,a)^\sim +(0,a)\mid a\in\cS\},\\
\cK(\cA,-)_{(\bar 1,\bar 0)}&=\{T\in \frinstrl(\cA,-)\mid T^\varepsilon=-T\}\oplus
		\{(0,a)^\sim -(0,a)\mid a\in\cS\},\\
\cK(\cA,-)_{(\bar 0,\bar 1)}&=\{(x,0)^\sim+(x,0)\mid x\in\cA\},\\
\cK(\cA,-)_{(\bar 1,\bar 1)}&=\{(x,0)^\sim-(x,0)\mid x\in\cA\}.
\end{split}
\]
The subspace $\{T\in \frinstrl(\cA,-)\mid T^\varepsilon=-T\}$ equals 
$\espan{V_{x,y}-V_{x,y}^\varepsilon\mid x,y\in\cA}=\espan{V_{x,y}+V_{y,x}\mid x,y\in\cA}$.
But $V_{x,y}+V_{y,x}=L_{x\overline{y}+y\overline{x}}$ for any $x,y\in\cA$, which gives:
\[
\{T\in \frinstrl(\cA,-)\mid T^\varepsilon=-T\}=L_\cH,
\]
where $\cH=\cH(\cA,-)=\{x\in\cA\mid \overline{x}=x\}$ is the subspace of symmetric elements in $\cA$.
Therefore, we have $\cK(\cA,-)_{(\bar 1,\bar 0)}=L_{\cH}\oplus\{(0,a)^\sim -(0,a)\mid a\in\cS\}$.

As in \cite[\S 4]{EldOku_S4} write, for $x\in\cA$,
\[
\left\{\begin{aligned}
 \epsilon_1(x)&=(x,0)+(x,0)^\sim\in\cK(\cA,-)_{(\bar 0,\bar 1)},\\
 \epsilon_2(x)&=(\overline{x},0)-(\overline{x},0)^\sim \in\cK(\cA,-)_{(\bar 1,\bar 1)}\\
 \epsilon_0(x)&=\frac{1}{2}\Bigl(L_{x+\overline{x}}
   +\bigl((0,x-\overline{x})-(0,x-\overline{x})^\sim\bigr)\Bigr)\in\cK(\cA,-)_{(\bar 1,\bar 0)}.
\end{aligned}\right.
\]
The bracket in \eqref{eq:KA-} of these elements in $\cK(\cA,-)$ behaves as follows, 
for any $x,y\in \cA$ (see
\cite[(4.2)]{EldOku_S4}, noting that all the computations there work also in characteristic $3$):
\[
\begin{split}
[\epsilon_0(x),\epsilon_1(y)]&=\epsilon_2(\overline{xy}),\\
[\epsilon_1(x),\epsilon_2(y)]&=-2\epsilon_0(\overline{xy}),\\
[\epsilon_2(x),\epsilon_0(y)]&=-\epsilon_1(\overline{xy}).
\end{split}
\]
(We refrain to use $-2=1$ here, because these formulas are valid in any characteristic $\neq 2$.)

With $\tilde\epsilon_0(x)=\epsilon_0(x)$, $\tilde\epsilon_1(x)=-\frac{1}{2}\epsilon_1(x)$, and
$\tilde\epsilon_2(x)=\epsilon_2(x)$, these last equations become
\begin{equation}\label{eq:tilde_epsilones}
\begin{split}
[\tilde\epsilon_0(x),\tilde\epsilon_1(y)]&=-\frac{1}{2}\tilde\epsilon_2(\overline{xy}),\\
[\tilde\epsilon_1(x),\tilde\epsilon_2(y)]&=\tilde\epsilon_0(\overline{xy}),\\
[\tilde\epsilon_2(x),\tilde\epsilon_0(y)]&=2\tilde\epsilon_1(\overline{xy}).
\end{split}
\end{equation}
Note that with $\gamma=(1,-1,2)$, $-\gamma_1\gamma_3^{-1}=-\frac{1}{2}$, 
$-\gamma_2\gamma_1^{-1}=1$, and $-\gamma_3\gamma_2^{-1}=2$, and hence there is a unique
Lie algebra homomorphism  as follows:
\[
\begin{split}
\Phi:\frstu_3(\cA,-,\gamma)&\longrightarrow \cK(\cA,-)\\
u_{12}[x]\ &\mapsto\ \tilde\epsilon_0(x)\\
u_{23}[x]\ &\mapsto\ \tilde\epsilon_1(x)\\
u_{31}[x]\ &\mapsto\ \tilde\epsilon_2(x)
\end{split}
\]
for any $x\in\cA$. Note that $\Phi$ is surjective because 
$\tilde\epsilon_1(\cA)+\tilde\epsilon_2(\cA)=\cK(\cA,-)_1\oplus\cK(\cA,-)_{-1}$, and this subspace
generates $\cK(\cA,-)$. Moreover, $\Phi$ is homogeneous of trivial degree for the 
$(\ZZ/2\ZZ)^2$-gradings on $\frstu_3(\cA,-,\gamma)$ and $\cK(\cA,-)$, being bijective on
$u_{12}[\cA]$, $u_{23}[\cA]$, and $u_{31}[\cA]$. As a consequence, the ideal $\ker\Phi$ is contained in
$\frstu_3(\cA,-,\gamma)_{(\bar 0,\bar 0)}=\frg$ and, therefore, 
$[\ker\Phi,u_{12}[\cA]]$ is contained in  both 
$[\frg,u_{12}[\cA]]\subseteq u_{12}[\cA]$ and in $\ker\Phi$, so it is trivial. It follows that
$\ker\Phi$ annihilates $u_{12}[\cA]$ and similarly $u_{23}[\cA]$ and $u_{31}[\cA]$, which generate
$\frstu_3(\cA,-,\gamma)$. Thus $\ker\Phi$ is contained in the center 
$Z\bigl(\frstu_3(\cA,-,\gamma)\bigr)$. On the other hand, the image under $\Phi$ of this center
is contained in the center of $\cK(\cA,-)$, which is trivial.

We summarize our findings in the next result.

\begin{proposition}\label{pr:stu3_KA-}
Let $(\cA,-)$ be a structurable algebra, then its Kantor algebra $\cK(\cA,-)$ is isomorphic, as a 
$(\ZZ/2\ZZ)^2$-graded Lie algebra, to
the quotient of the Lie algebra
$\frstu_3(\cA,-,\gamma)$ by its
center, with $\gamma=(1,-1,2)$.
\end{proposition}

\smallskip

Assume now that $\cA=\cC_1\otimes\cC_2$ for a Cayley algebra $\cC_1$ and a unital composition
algebra $\cC_2$. Consider the so called \emph{para-Hurwitz} algebras $\overline{\cC}_1^{\op}$ and 
$\overline{\cC}_2^{\op}$ attached to the opposite
algebras $\cC_1^{\op}$ and $\cC_2^{\op}$. That is, the multiplication in $\overline{\cC}_1^{\op}$ is 
given by $x_1\bullet y_1=\overline{y_1}\,\overline{x_1}=\overline{x_1y_1}$ for $x_1,y_1\in\cC_1$,
 and similarly for $\overline{\cC}_2^{\op}$.

These para-Hurwitz algebras are examples of \emph{symmetric composition algebras} and, as shown
in \cite[Remark 3.2]{Elduque_Magic}, for each
$\alpha=(\alpha_0,\alpha_1,\alpha_2)\in(\FF^\times)^3$, there
is a Lie algebra $\frg_{\alpha}(\overline{\cC}_1^{\op},\overline{\cC}_2^{\op})$ defined on the
vector space
\[
\bigl(\tri(\overline{\cC}_1^{\op})\oplus\tri(\overline{\cC}_2^{\op})\bigr)\oplus 
\Bigl(\oplus_{i=0}^2\iota_i(\overline{\cC_1}^{\op}\otimes\overline{\cC}_2^{\op})\Bigr),
\]
where $\iota_i(\cC_1\otimes \cC_2)$ is a copy of 
$\cC_1\otimes\cC_2=\cA$, $\tri(\overline{\cC}_1^{\op})$ is the triality Lie algebra of 
$\overline{\cC}_1^{\op}$, whose 
elements are the triples $(d_0,d_1,d_2)\in\frso(\cC_1,\nup_1)$ such that 
\[
d_0(x_1\bullet y_1)=d_1(x_1)\bullet y_1+x_1\bullet d_2(y_1)
\]
for any $x_1,y_1\in\cC_1$, and similarly for $\tri(\overline{\cC}_2^{\op})$. The bracket in
$\frg_{\alpha}(\overline{\cC}_1^{\op},\overline{\cC}_2^{\op})$ satisfies
\[
[\iota_0(x_1\otimes x_2),\iota_1(y_1\otimes y_2)]
=\alpha_2\iota_2\bigl((x_1\bullet y_1)\otimes(x_2\bullet y_2)\bigr)
=\alpha_2\iota_2\bigl(\overline{x_1y_1}\otimes\overline{x_2y_2}\bigr),
\]
and cyclically on $0,1,2$. Then, with $\alpha=(1,2,-\frac{1}{2})$, this last equation works exactly as
\eqref{eq:tilde_epsilones}, which shows that there is a Lie algebra homomorphism determined as follows:
\[
\begin{split}
\Psi:\frstu_3(\cA,-,\gamma)
              &\longrightarrow \frg_{\alpha}(\overline{\cC}_1^{\op},\overline{\cC}_2^{\op})\\
u_{12}[x\otimes y]\ &\mapsto\ \iota_0(x\otimes y)\\
u_{23}[x\otimes y]\ &\mapsto\ \iota_1(x\otimes y)\\
u_{31}[x\otimes y]\ &\mapsto\ \iota_2(x\otimes y)
\end{split}
\]
for any $x\in\cC_1$ and $y\in\cC_2$. The image of $\Psi$ is the subalgebra of
$\frg_{\alpha}(\overline{\cC}_1^{\op},\overline{\cC}_2^{\op})$ generated by 
$\sum_{i=0}^2\iota_i(\cC_1^{\op}\otimes\cC_2^{\op})$, which is the derived subalgebra
$\frg_{\alpha}(\overline{\cC}_1^{\op},\overline{\cC}_2^{\op})^{(1)}$. 

The same arguments as above give our next result.

\begin{proposition}\label{pr:stu3_gCC}
Let $(\cA,-)$ be the structurable algebra obtained as the tensor product of the Cayley algebra $\cC_1$ 
and a unital composition algebra $\cC_2$. Then the Lie algebra 
$\frg_{\alpha}(\overline{\cC}_1^{\op},\overline{\cC}_2^{\op})^{(1)}$, with 
$\alpha=(1,2,-\frac{1}{2})$, is isomorphic, as a $(\ZZ/2\ZZ)^2$-graded Lie algebra, to the quotient of 
$\frstu_3(\cA,-,\gamma)$ by its
center, with $\gamma=(1,-1,2)$.
\end{proposition}

\begin{corollary}\label{co:KA-_gCC}
Let $(\cA,-)$ be the structurable algebra obtained as the tensor product of the Cayley algebra $\cC_1$ 
and a unital composition algebra $\cC_2$. Then its Kantor algebra $\cK(\cA,-)$ is isomorphic, 
as a $(\ZZ/2\ZZ)^2$-graded Lie algebra, to the
Lie algebra $\frg_{\alpha}(\overline{\cC}_1^{\op},\overline{\cC}_2^{\op})^{(1)}$, with
$\alpha=(1,2,-\frac{1}{2})$.
\end{corollary}

\smallskip

\noindent\emph{Proof of Theorem \ref{th:instrl(C1C2)}}.\quad
To begin with, given a structurable algebra $(\cA,-)$, the dimension of the associated Lie algebra
$\cK(\cA,-)$ is $\dim\frinstrl(\cA,-)+2\dim\cA+2\dim\cS$. Hence, for $\cA=\cC_1\otimes\cC_2$ for a Cayley
algebra $\cC_1$ and a unital composition algebra $\cC_2$ of dimension $n$ ($n=1,2,4$ or $8$), we get
\begin{multline*}
\dim\cK(\cA,-)=\dim\frinstrl(\cA,-)+2\times 8\times n+2\times (7+n-1)
\\=
     	\dim\frinstrl(\cA,-)+18n+12.
\end{multline*}
On the other hand, the dimension of 
$\frg_{\alpha}(\overline{\cC}_1^{\op},\overline{\cC}_2^{\op})^{(1)}$ is $52$, $77$, $133$, or $248$,
depending on $n$ being $1$, $2$, $4$ or $8$ (see \cite[\S 3]{Elduque_Magic}). This gives that the
dimension of $\frinstrl(\cA,-)$ is $22$, $29$, $49$, or $92$, respectively, while the dimension of 
$L_\cS L_\cS$ is, according to Lemma \ref{le:LSLS}, $22$, $29$, $46$, or $92$.

We conclude that if the dimension of $\cC_2$ is not $4$, these dimensions coincide, and hence the
ideal $L_\cS L_\cS$ is the whole $\frinstrl(\cA,-)$. However, if the dimension of $\cC_2$ is $4$, then
$\cC_2$ is isomorphic to the algebra of $2\times 2$-matrices over $\FF$ (recall that we are assuming in 
this section that the field $\FF$ is algebraically closed). In this case, for any $0\neq a\in\cS_2$, there
are elements $b,c\in\cS_2$ such that $a=bc$ (exercise for the reader!) and, because of the associativity of $\cC_2$, $L_a=L_bL_c$
lies in $L_\cS L_\cS\subseteq \frinstrl(\cA,-)$. But $V_{a,1}=L_a+2R_a$, so we conclude that the
right multiplication $R_a$ also belongs to $\frinstrl(\cA,-)$. Besides,
 $R_a^\delta=R_a+R_{\overline{R_a(1)}}=R_a-R_a=0$ (see Proposition \ref{pr:KantorLieAlgebra}).
Hence, $R_{\cS_2}$ lies in $\ker\delta$, while $\delta$ is one-to-one on $L_\cS L_\cS$ 
(Lemma \ref{le:LSLS}). Therefore, $R_{\cS_2}\cap L_\cS L_\cS=0$ and, by dimension count, 
$\frinstrl(\cA,-)=R_{\cS_2}\oplus L_\cS L_\cS$. \qed

\medskip

\subsection{The Lie algebra \texorpdfstring{$S(\cA,\cA)$}{S(A,A)}}

Let $(\cA,-)$ be the structurable algebra obtained as the tensor product of the Cayley algebra $\cC_1$ 
and a unital composition algebra $\cC_2$. Fix an element $s\in\cS$ with $L_s$ invertible, and take
the element $t\in\cS$ with $L_sL_t=L_tL_s=\id$. Equation \eqref{eq:AllQ1} shows that $Q(a)\neq 0$
and $t=-\frac{1}{Q(a)}s^{\sharp}$, where $Q$ is the Albert form.

Recall from Theorem \ref{th:structurable_Jternary}
that $\cA$ becomes a $J$-ternary algebra with triple product $\ptriple{x,y,z}=V_{x,sy}(z)$ for
$x,y,z\in\cA$, and that the Lie algebra $S(\cA,\cA)$ consists of those elements $T\in\frinstrl(\cA,-)$
such that $T^\delta(t)=0$. Because of \eqref{eq:KA-}, this can be expressed in terms
of the associated Lie algebra $\cK(\cA,-)$ in Proposition \ref{pr:KantorLieAlgebra} as follows:
\[
S(\cA,\cA)=\{T\in\frinstrl(\cA,-)\mid [T,(0,t)]=0\}.
\]
As in the proof of Theorem \ref{th:structurable_Jternary}, consider the elements $E=(0,t)$, 
$F=(0,s)^\sim$ and $H=\id$ in $\cK(\cA,-)$. Denote by $\frsl_2$ the subalgebra of $\cK(\cA,-)$ spanned
by $E$, $F$, and $H$. Hence, an element $T\in\frinstrl(\cA,-)$ lies in
$S(\cA,\cA)$ if and only if $[E,T]=0$. By Lemma \ref{le:5graded} this is equivalent to centralizing 
the subalgebra $\frsl_2$:
\[
S(\cA,\cA)=\Cent_{\cK(\cA,-)}(\frsl_2).
\]
Consider the subspace $\cS'\bydef (\FF t)^\perp$ (orthogonal relative to the Albert form $Q$ in
Proposition \ref{pr:AlbertForm}). Lemma \ref{le:LSLS} shows that the representation 
$\delta\colon\mathfrak{instrl}(\cA,-)\rightarrow\End_{\FF}(\cS)$ induces
an isomorphism of Lie algebras:
\[
\begin{split}
\delta'\colon M_{\cS',\cS'}&\longrightarrow \frso(\cS',Q)\\
          T\ &\mapsto\ T^\delta\vert_{\cS'},
\end{split}
\]
where, as usual, $M_{\cS',\cS'}=\espan{M_{a,b}=L_aL_{b^\sharp}-L_bL_{a^\sharp}\mid a,b\in\cS'}$.

A straightforward consequence of Theorem \ref{th:instrl(C1C2)} is the following result:

\begin{corollary}\label{co:MS'S'}
Let $(\cA,-)$ be the structurable algebra obtained as the tensor product of the Cayley algebra $\cC_1$ 
and a unital composition algebra $\cC_2$. Fix elements $s,t\in\cS$ such that $L_sL_t=L_tL_s=\id$, and
consider the subspace $\cS'$ of $\cS$ orthogonal to $t$ relative to the Albert form $Q$. 
\begin{itemize}
\item
If the dimension of $\cC_2$ is $1$, $2$, or $8$, then the Lie subalgebra $S(\cA,\cA)$ in Theorem
\ref{th:structurable_Jternary} equals $M_{\cS',\cS'}$, which is isomorphic to the orthogonal Lie algebra
$\frso(\cS',Q)$.
\item
If the dimension of $\cC_2$ is $4$,  then the Lie subalgebra $S(\cA,\cA)$ in Theorem
\ref{th:structurable_Jternary} equals $R_{\cS_2}\oplus M_{\cS',\cS'}$, which is the direct sum of 
the three-dimensional simple Lie algebra $R_{\cS_2}$, isomorphic to $\frsl_2$, and the subalgebra
$M_{\cS',\cS'}$, which is isomorphic to the orthogonal Lie algebra $\frso(\cS',Q)$.
\end{itemize}
\end{corollary}

\medskip

\subsection{Structure of \texorpdfstring{$\cA$}{A} as a module for
 \texorpdfstring{$S(\cA,\cA)$}{S(A,A)}}

In order to determine the Lie superalgebra $\Lss(\cA)=S(\cA,\cA)\oplus\cA$ attached to the
 $J$-ternary algebra associated to 
a structurable algebra $(\cA,-)$, 
obtained as the tensor product of a Cayley algebra $\cC_1$ and a unital composition algebra $\cC_2$, 
and a fixed element $s\in\cS$ with $L_s$ invertible, we need to know the structure of $\cA$ as a module
for $S(\cA,\cA)$.

We will follow now some arguments in \cite{Allison_Tensor}. Consider the linear map
\[
\begin{split}
\Phi\colon \cS'&\longrightarrow \End_\FF(\cA)\\
   a\,&\mapsto\ L_aL_s,
\end{split}
\]
which satisfies
\[
\Phi(a)^2=L_aL_sL_aL_s=Q(a)L_{s^\sharp}L_s=-Q(a)Q(s)\id,
\]
for any $a\in\cS'$, because of \eqref{eq:AllQ3} and \eqref{eq:AllQ1}. Hence, $\Phi$ extends to
an algebra homomorphism, still denoted by $\Phi$:
\begin{equation}\label{eq:PhiClifford}
\Phi\colon \Cl(\cS',\tilde Q)\longrightarrow \End_\FF(\cA),
\end{equation}
where $\tilde Q$ is the quadratic form defined on $\cS'$ by $\tilde Q(a)=-Q(a)Q(s)$ and 
$\Cl(\cS',\tilde Q)$ denotes the associated Clifford algebra. Denote by $u\cdot v$ the multiplication in
$\Cl(\cS',\tilde Q)$. The Lie algebra $\frso(\cS',Q)=\frso(\cS',\tilde Q)$ lives inside $\Cl(\cS',\tilde Q)$
as the subspace $\espan{[a,b]^\cdot\bydef a\cdot b-b\cdot a\mid a,b\in\cS'}$.

Actually, for any $a,b,c\in\cS'$, $a\cdot b\cdot a=a\cdot b\cdot a+a\cdot a\cdot b-\tilde Q(a)b=
a\cdot (a\cdot b+b\cdot a)-\tilde Q(a)b=\tilde Q(a,b)a-\tilde Q(a)b$, so that, by linearization, we
have $a\cdot b\cdot c+c\cdot b\cdot a=\tilde Q(a,b)c+\tilde Q(c,b)a-\tilde Q(a,c)b$. As a consequence,
\begin{multline*}
[a\cdot b-b\cdot a,c]^\cdot=(a\cdot b\cdot c+c\cdot b\cdot a)-(b\cdot a\cdot c+c\cdot a\cdot b)\\
	=-2\bigl(\tilde Q(a,c)b-\tilde Q(b,c)a\bigr)=-2\sigma_{a,b}^{\tilde Q}(c).
\end{multline*}
In other words, the linear map $\iota:\frso(\cS',\tilde Q)\rightarrow \Cl(\cS',\tilde Q)$ such that
\[
\iota(\sigma_{a,b}^{\tilde Q})=-\frac{1}{2}(a\cdot b-b\cdot a),
\] 
for $a,b\in\cS'$, is a one-to-one
homomorphism of Lie algebras.

\begin{lemma}\label{le:MS'S'Cl}
The diagram
\[
\begin{tikzcd}
\frso(\cS',\tilde Q) \arrow[d, "\iota"'] & M_{\cS',\cS'} 
\arrow[l, "{\delta'}", "\simeq"'] \arrow[d, hook] \\
\Cl(\cS',\tilde Q) \arrow[r, "\Phi"] & \End_\FF(\cA)
\end{tikzcd}
\]
commutes.
\end{lemma}
\begin{proof}
For $a,b\in\cS'$ we have
\[
\Phi(a\cdot b-b\cdot a)=L_aL_sL_bL_s-L_bL_sL_aL_s
  =Q(s)(L_aL_{b^\sharp}-L_bL_{a^\sharp})=Q(s)M_{a,b},
\]
where we have used \eqref{eq:AllQ3} and that 
$Q(s,b^{\sharp})=Q(s^{\sharp},b)=-Q(s)Q(t,b)=0=Q(s,a^{\sharp})$, and
\[
\begin{split}
\delta'(M_{a,b})&=2\sigma_{a,b}^Q=2\bigl(Q(a,.)b-Q(b,.)a\bigr)\quad\text{by \eqref{eq:Mab_sigma}}\\
 &=-\frac{2}{Q(s)}\bigl(\tilde Q(a,.)b-\tilde Q(b,.)a\bigr)=-\frac{2}{Q(s)}\sigma_{a,b}^{\tilde Q}.
\end{split}
\]
Then we compute
\[
\Phi\circ\iota\circ\delta'(M_{a,b})
	=-\frac{2}{Q(s)}\Phi\circ\iota(\sigma_{a,b}^{\tilde Q})\\
	=\frac{1}{Q(s)}\Phi(a\cdot b-b\cdot a)=M_{a,b},
\]
as required.
\end{proof}

We will consider the different possibilities, according to the dimension of $\cC_2$:

\bigskip

\noindent $\mathbf{\dim\cC_2=1}$:\quad In this case we get the following dimensions:
\[
\dim\cS=7,\quad\dim\cS'=6,\quad\dim\End_\FF(\cA)=8^2=2^6=\dim\Cl(\cS',\tilde Q),
\]
so that the homomorphism $\Phi$ in \eqref{eq:PhiClifford} is an isomorphism, because $\Cl(\cS',\tilde Q)$
is a simple (associative) algebra. Then $\cA$ is the sum of the two half-spin modules for the
orthogonal Lie algebra $M_{\cS',\cS'}\simeq\frso(\cS',\tilde Q)\simeq\frso_6\simeq\frsl_4$. These
half-spin modules, thought as modules for $\frsl_4$, are the natural four-dimensional module for 
$\frsl_4$ and its dual.

Therefore,  Corollary \ref{co:MS'S'} shows that the even part of $\Lss(\cA)$ is, up to isomorphism, 
$\frsl_4\simeq\frsl(W)$ for a four-dimensional vector space $W$, while the odd part is isomorphic,
as a module for $\frsl(W)$ to $W\oplus W^*$. But, as is well-known (and valid in characteristic
$3$), one has
$
\Hom_{\frsl(W)}(W\otimes W,\frsl(W))=0=\Hom_{\frsl(W)}(W^*\otimes W^*,\frsl(W))
$
and 
$
\Hom_{\frsl(W)}(W\otimes W^*,\frsl(W))$ is one-dimensional. Hence, the Lie bracket in
$\Lss(\cA)\subuno\times\Lss(\cA)\subuno\rightarrow \Lss(\cA)\subo$ is uniquely determined, and hence
$\Lss(\cA)$ is isomorphic to the projective special linear Lie superalgebra $\frpsl(4\vert 1)$.

\bigskip

\noindent $\mathbf{\dim\cC_2=2}$:\quad Here we get
 the following dimensions:
\[
\dim\cS=8,\quad\dim\cS'=7,\quad \dim\Cl(\cS',\tilde Q)=2^7,\quad\dim\Cl\subo(\cS',\tilde Q)=2^6.
\]
The even Clifford algebra is simple of dimension $2^6$, so its unique irreducible module has dimension
 $8$, and hence the representation of 
$S(\cA,\cA)=M_{\cS',\cS'}$ on $\cA$, given by the composition in Lemma \ref{le:MS'S'Cl}, is the direct sum
of two copies of the spin module for $M_{\cS',\cS'}\simeq \frso_7$. But if $W$ is the spin module for
$\frso_7$, then $\Hom_{\frso_7}(W\otimes W,\frso_7)$ is one-dimensional and spanned by a 
skew-symmetric map (see, e.g., \cite[Proposition 2.12]{Elduque_SomeNew}). Therefore, up to
scalars, there is a unique $\frso_7$-invariant symmetric bilinear map 
$\Lss(\cA)\subuno\times\Lss(\cA)\subuno\rightarrow\Lss(\cA)\subo$, and $\Lss(\cA)$ is necessarily
isomorphic to the Lie superalgebra $\frg(3,3)=\frg(S_{1,2},S_{1,2})$ in 
\cite[Proposition 5.19]{CunhaElduque_Extended}.

\bigskip

\noindent $\mathbf{\dim\cC_2=4}$:\quad 
In this case we get the following dimensions:
\[
\dim\cS=10,\quad\dim\cS'=9,\quad \dim\Cl(\cS',\tilde Q)=2^9,\quad\dim\Cl\subo(\cS',\tilde Q)=2^8.
\]
Corollary \ref{co:MS'S'} shows that $S(\cA,\cA)$ is the direct sum of the simple three-dimensional
Lie algebra $R_{\cS_2}$ and the ideal $M_{\cS',\cS'}\simeq\frso_9$. The action of the quaternion
algebra $\cC_2\simeq 1\otimes\cC_2$ by multiplication on the right on $\cA$ commutes with the
action of $M_{\cS',\cS'}$, because of the associativity of $\cC_2$. By dimension count, and using
that the even Clifford algebra $\Cl\subo(\cS',\tilde Q)$ is simple, $\Phi$ induces
an isomorphism, denoted by the same letter:
\[
\Phi:\Cl\subo(\cS',\tilde Q)\rightarrow \End_{\cC_2}(\cC_1\otimes\cC_2)
	\simeq \End_\FF(\cC_1)\otimes\cC_2.
\]
As a module for $M_{\cS',\cS'}\simeq\frso_9$, $\cA$ is a sum of two copies of the spin  module $W$ 
(the irreducible module for the simple algebra $\Cl\subo(\cS',\tilde Q)$), and thus 
$R_{\cS_2}$ lies in $\End_{M_{\cS',\cS'}}(\cA)\simeq\End_{\frso_9}(W\oplus W)\simeq \Mat_2(\FF)$.
As a consequence, $\Lss(\cA)\subo$ is isomorphic to $\frsl_2\oplus \frso_9$, while $\Lss(\cA)\subuno$ is,
as a module for the even part, isomorphic to the tensor product of the two-dimensional natural module 
$V$ for $\frsl_2$ and the spin module $W$ for $\frso_9$. The uniqueness, up to scalars, of the elements 
in the
vector spaces $\Hom_{\frsl_2}(V\otimes V,\FF)$, $\Hom_{\frsl_2}(V\otimes V,\frsl_2)$, 
$\Hom_{\frso_9}(W\otimes W,\FF)$, and $\Hom_{\frso_9}(W\otimes W,\frso_9)$ (see
\cite[Proposition 2.12]{Elduque_SomeNew}) forces that the bracket 
$\Lss(\cA)\subuno\times\Lss(\cA)\subuno\rightarrow\Lss(\cA)\subo$ is uniquely determined up to scalars
(because of the Jacobi identity), and hence $\Lss(\cA)$ turns out to be isomorphic to the exceptional
Lie superalgebra $\frel(5;3)$ (\cite[Theorem 3.3]{Elduque_Models}).

\bigskip

\noindent $\mathbf{\dim\cC_2=8}$:\quad 
Here we have the following dimensions:
\[
\dim\cS=14,\quad\dim\cS'=13,\quad \dim\Cl\subo(\cS',\tilde Q)=2^{12}=(2^6)^2=\dim\End_\FF(\cA),
\]
and hence the homomorphism $\Phi$ in \eqref{eq:PhiClifford} restricts to an isomorphism from
the even Clifford algebra $\Cl\subo(\cS',\tilde Q)$, which is simple, onto $\End_\FF(\cA)$. It turns out
that $\cA$ is the spin module for $S(\cA,\cA)=M_{\cS',\cS'}\simeq \frso_{13}$. As the space
$\Hom_{\frso_{13}}(W\otimes W,\frso_{13})$ is one-dimensional 
(see \cite[Proposition 2.12 and Theorem 3.1]{Elduque_SomeNew}), where $W$ is the spin module
for $\frso_{13}$, we conclude that $\Lss(\cA)$ is isomorphic to the Lie superalgebra 
$\frg(6,6)=\frg(S_{4,2},S_{4,2})$ in \cite[Proposition 5.10]{CunhaElduque_Extended}.

\bigskip

We summarize our findings in the next theorem.

\begin{theorem}
Let $(\cA,-)$ be the structurable algebra obtained as the tensor product of the Cayley algebra $\cC_1$ 
and a unital composition algebra $\cC_2$ over an algebraically closed field of characteristic $3$. 
Let $s$ be a skew-symmetric element such that  $L_s$ is invertible, and let
$\Lss(\cA)$ be the Lie superalgebra attached to the $J$-ternary algebra defined on $\cA$ in Theorem \ref{th:structurable_Jternary}.  Then the Lie superalgebra $\Lss(\cA)$ is given, up to isomorphism, by
the following table:
\[
\begin{tabular}{c|cccc}
$\dim\cC_2$ &{\vrule height 12pt width 0ptdepth 2pt} $1$ & $2$ & $4$ & $8$\\
\hline
$\Lss(\cA)$ &{\vrule height 12pt width 0pt depth 2pt}
                       $\frpsl(4\vert 1)$&$\frg(3,3)$&$\frel(5;3)$&$\frg(6,6)$
\end{tabular}
\]
\end{theorem}

\medskip

\subsection{A magic square of Lie superalgebras}\label{ss:magic}

If our structurable algebra $(\cA,-)$ is the tensor product of two associative unital composition algebras,
then $(\cA,-)$ is associative with involution and hence it fits in Example \ref{ex:hermitian_prototypical}
with $\cW=0$. According to this Example, the isotope $\cB\bydef\cA^{(s)}$ is endowed with
the involution $\tau:x\mapsto -\overline{x}$, $\cA$ is a left $\cB$-module by means of 
$b\bullet x\bydef bsx$, for all $b,x\in\cA$, and it is endowed with a nondegenerate skew-hermitian form 
$\tilde\hup:\cA\times\cA\rightarrow \cB$ given by $\tilde\hup(x,y)=-x\overline y$. The
triple product on $\cA$ is given by
\[
\ptriple{x,y,z}=V_{x,sy}(z)=\tilde\hup(x,y)sz+\tilde\hup(z,y)sx+\tilde\hup(z,x)sy,
\]
for $x,y,z\in\cA$.

In case $\dim\cC_1=\dim\cC_2=4$, $\cA$ (and hence $\cB$) is a simple algebra, so that it is
isomorphic to $\Mat_4(\FF)$, the involution $-$ on 
$\cA$ is orthogonal, and hence the involution
$\tau$ on $\cB$ is symplectic, and Theorem \ref{th:proto_super} implies that $\Lss(\cA)$ is isomorphic
to $\frosp(4\vert 4)$. The case of $\dim\cC_1=1$, $\dim\cC_2=4$ is similar, but $\cB$ is isomorphic
to $\Mat_2(\FF)$ in this case, the involution $\tau$ is orthogonal, and $\Lss(\cA)$ is isomorphic
to $\frosp(2\vert 2)$.

For $\dim\cC_1=2$ and $\dim\cC_2=4$, $(\cA,-)$ is a simple algebra with involution, but it is not simple
as an algebra, the involution being of the second kind, and Theorem \ref{th:proto_super} gives that
$\Lss(\cA)$ is isomorphic to $\frpsl(2\vert 2)$. The same happens for $\dim\cC_1=2$ and $\dim\cC_2=1$,
in which case $\Lss(\cA)$ is isomorphic to $\frpsl(1\vert 1)$.

The case of  $\dim\cC_1=\dim\cC_2=1$ does not give a $J$-ternary algebra, as $\cS=0$. Finally, if
$\dim\cC_1=\dim\cC_2=2$, $(\cA,-)$ is the direct sum of two two-dimensional ideals, both with involution
of the second kind, so it follows that $\Lss(\cA)$ is isomorphic to $\frpsl(1\vert 1)\oplus\frpsl(1\vert 1)$.

\smallskip

We finish the paper by collecting all this information
 in the following \emph{magic square} of Lie superalgebras:
\[
\begin{matrix}
&\quad \dim\cC_1\\
\begin{matrix}\dim\cC_2\\[45pt] \null\end{matrix}&
\vbox{\offinterlineskip
 \halign{\hfil\ $#$\quad \hfil&%
 \vreglon #%
 &\hfil\ $#$\ \hfil&\hfil\ $#$\ \hfil
 &\hfil\ $#$\ \hfil&\hfil\ $#$\ \hfil&%
 \vrule  depth 4pt width .5pt #\cr
 & &1&2&4&8&\omit\cr
 \multispan7{\hreglonfill}\cr
  1&&&\frpsl(1\vert 1)&\frosp(2\vert 2)&\frpsl(4\vert 1)&\cr
  2&&\frpsl(1\vert 1)&\frpsl(1\vert 1)\oplus\frpsl(1\vert 1)&\frpsl(2\vert 2)&\frg(3,3)&\cr
  4&&\frosp(2\vert 2)&\frpsl(2\vert 2)&\frosp(4\vert 4)&\frel(5;3)&\cr
  8&&\frpsl(4\vert 1)&\frg(3,3)&\frel(5;3)&\frg(6,6)&\cr
 &\multispan6{\hregletafill}\cr}}
\end{matrix}
\]
\vspace*{-30pt}

\noindent that contains three of the exceptional finite-dimensional Lie superalgebras specific of characteristic $3$.


\end{document}